\documentclass[reqno,twoside,12pt]{amsart}
\usepackage{amssymb,amsfonts,amsthm,amsmath,mathrsfs}
\usepackage{enumitem}
\usepackage{mathtools}
\mathtoolsset{showonlyrefs}
\usepackage[pagebackref=false]{hyperref}
\usepackage[hmargin=1in,vmargin=1in]{geometry}
\usepackage{cite}

\def\eqdef{\stackrel{\rm def}{=}}
\def\d{{\rm d}}
\def\beq{\begin{equation}}
\def\eeq{\end{equation}}
\def\beqs{\begin{equation*}}
\def\eeqs{\end{equation*}}

\newcommand{\ddt}{\frac{\d}{\d t}}

\newcommand{\varep}{\varepsilon}

\renewcommand{\Re}{\operatorname{Re}}
\renewcommand{\Im}{\operatorname{Im}}

\newcommand{\LL}{\mathcal L}
\newcommand{\iln}{L}
\newcommand{\C}{\mathbb{C}}

\newcommand{\K}{\mathbb{K}}
\newcommand{\R}{\mathbb{R}}
\newcommand{\N}{\mathbb{N}}
\newcommand{\Z}{\mathbb{Z}}

\newcommand{\bigo}{\mathcal{O}}

\newcommand{\tnum}{\rm(\roman*)}
\newcommand{\rnum}{\rm(\alph*)}

\newtheorem{theorem}{Theorem}[section]
\newtheorem{lemma}[theorem]{Lemma}
\newtheorem{proposition}[theorem]{Proposition}

\newtheorem{assumption}[theorem]{Assumption}
\newtheorem{definition}[theorem]{Definition}

\theoremstyle{definition}
\newtheorem{example}[theorem]{Example}
\newtheorem{scenario}[theorem]{Scenario}
\newtheorem{remark}[theorem]{Remark}

\theoremstyle{remark}

\usepackage{color}

\definecolor{darkred}{rgb}{.70,.12,.20}

\definecolor{darkgreen}{rgb}{.20,.52,.14}

\date{\today}
\numberwithin{equation}{section}

\title{Asymptotic expansions about infinity for solutions of nonlinear differential equations with coherently decaying forcing functions}

\author{Luan Hoang}

\address{Department of Mathematics and Statistics,
Texas Tech University\\
1108 Memorial Circle, Lubbock, TX 79409--1042, U. S. A.}
\email{luan.hoang@ttu.edu}
\thanks{\textit{Acknowledgement.}  The author would like to thank Dat Cao for many valuable discussions.}

\makeatletter
\@namedef{subjclassname@2020}{
  \textup{2020} Mathematics Subject Classification}
\makeatother

\begin{document}
 
\begin{abstract}
This paper studies, in fine details, the long-time asymptotic behavior of decaying solutions of a general class of dissipative systems of nonlinear differential equations in complex Euclidean spaces. The forcing functions decay, as time tends to infinity, in a coherent way expressed by combinations of the exponential, power, logarithmic and iterated logarithmic functions. The decay may contain sinusoidal oscillations not only in time but also in the logarithm and iterated logarithm of time. It is proved that the decaying solutions admit corresponding asymptotic expansions, which can be constructed concretely. In the case of the real  Euclidean spaces, the real-valued decaying solutions are proved to admit real-valued asymptotic expansions. Our results unite and extend the theory investigated  in many previous works.
\end{abstract}

\maketitle
\tableofcontents
 
\pagestyle{myheadings}\markboth{L. Hoang}
{Nonlinear Differential Equations  with Coherently Decaying Forcing Functions}

\section{Introduction}
This paper broadens the theory of the asymptotic expansions, as time tends to infinity, for solutions of systems of nonlinear ordinary differential equations (ODEs) that was studied in previous work \cite{CaH3}. The paper \cite{CaH3} itself was motivated by Foias--Saut's work   \cite{FS87} for the Navier--Stokes equations (NSE). 
In Foias and Saut's paper  \cite{FS87}, the NSE are written in the functional form as an ODE  
\beq\label{NSE}
u_t+Au+B(u,u)=0,
\eeq
where $A$ is a linear operator, and $B$ is a bi-linear form.
They proved that every solution $u(t)$ of \eqref{NSE} admits the following asymptotic expansion, as $t\to\infty$,
\beq\label{nsexp}
u(t)\sim \sum_{k=1}^\infty q_k(t)e^{-\mu_k t},
\eeq
where $q_k(t)$'s are polynomials in $t$ valued in some functional spaces, and $\mu_k$'s are positive numbers increasing strictly to infinity. Briefly speaking,  the solution $u(t)$ can be approximated, as $t\to\infty$, by the finite sum 
\beqs
s_N(t):=\sum_{k=1}^N q_k(t)e^{-\mu_k t},\text{ for }N\in\N,
\eeqs 
in the sense that the remainder $u(t)-s_N(t)$, as $t\to\infty$,  decays exponentially faster than the fastest decaying mode $e^{-\mu_N t}$ in $s_N(t)$. Because $\mu_k\nearrow\infty$, the asymptotic expansion \eqref{nsexp} provides very fine details about the asymptotic behavior of $u(t)$.

The result \eqref{nsexp} has been developed in many directions. One can study its associated normalization map, normal form, invariant nonlinear spectral manifold, etc., see \cite{FS84a,FS84b,FS91,FHS1,FHOZ1,FHOZ2,FHN1,FHN2} and references therein. The interested reader is referred to the survey \cite{FHS2} for more information on the subject.
A similar asymptotic expansion is obtained for the NSE of rotating fluids \cite{HTi1}. 
The asymptotic expansions for the associated Lagrangian trajectories are derived in \cite{H4} based on \eqref{nsexp}.

The asymptotic expansions of the type \eqref{nsexp} can also be established for other partial differential equations (PDEs) and ODEs. They are obtained in \cite{Shi2000} for a class of PDEs including dissipative wave equations, in \cite{Minea} for a system of analytic ODEs, and in \cite{CaHK1} for non-smooth ODEs. These results are for homogeneous equations which have more general nonlinearity than the bi-linear form in \eqref{NSE}. 

Concerning with inhomogeneous equations instead,  the papers \cite{HM2,CaH1,CaH2} study the NSE \eqref{NSE} with a forcing function $f(t)$ added to its right-hand side. It is proved that if  $f(t)$ decays to zero in a coherent way, then any solution $u(t)$ admits  an asymptotic expansion with a correspondingly coherent decay \cite{HM2,CaH1,CaH2}.

For more general inhomogeneous systems of ODEs, paper \cite{CaH3} investigates 
\beq\label{sys}
 y'=-Ay + G(y)+f(t) \text{ in }\R^n,
 \eeq
 where $A$ is a real diagonalizable $n\times n$ real matrix with positive eigenvalues, and $G(y)$ has the Taylor's expansion about the origin starting the quadratic monomials.
 It proves that if 
\beq\label{fyex}
f(t)\sim \sum p_k(\phi(t))\psi(t)^{-\gamma_k},\text{ then }
y(t)\sim \sum q_k(\phi(t))\psi(t)^{-\gamma_k},
\eeq
where, roughly speaking, $p_k$'s and $q_k$'s are functions of the same type.
For example, 
\begin{itemize}
\item $\psi(t)=e^t$, $\phi(t)=t$, $p_k$'s and $q_k$'s are polynomials. This yields the same asymptotic expansions as \eqref{nsexp}.
\item  $\psi(t)=t$, $\phi(t)=(\ln t,\ln\ln t)$, $p_k$'s and $q_k$'s are real power functions of two variables.
\item $\psi(t)=\ln t$, $\phi(t)=(\ln\ln t,\ln\ln\ln t,\ln\ln\ln\ln t)$, $p_k$'s and $q_k$'s are real power functions of three variables.
\end{itemize}

Although the asymptotic expansions in \eqref{fyex} are rather complicated, some common elements are still missing. For instance, because of the assumptions on the matrix $A$ and the real powers in the functions $p_k$'s and $q_k$'s, no oscillations  are  present in \eqref{fyex}. This paper will overcome this deficiency.
In fact, its new features are the following:
\begin{enumerate}[label=\rnum]
\item \label{new1} Studying \eqref{sys} in both $\C^n$ and $\R^n$.
\item  \label{new2} Treating a general matrix $A$, namely, $A$ is only required to have eigenvalues with positive real parts.
\item \label{new3} Allowing $p_k$'s and $q_k$'s to have complex powers, therefore, allowing  the asymptotic expansions  in \eqref{fyex} to have sinusoidal oscillations not only in $t$ but also in $\ln t$, $\ln\ln t$, etc. 
\item \label{new4} Providing more general but still concrete constructions for the $q_k$'s.
\end{enumerate}


Note that previous papers \cite{Minea,Shi2000} consider complex exponential functions and, hence,  already obtain asymptotic expansions  with sinusoidal functions $\cos(\omega t)$ and $\sin(\omega t)$. These functions are both oscillating and periodic. (See also \cite{HTi1} when the oscillation occurs due to the Coriolis effect.)
However, \ref{new3}	will allow the asymptotic expansions to contain functions such as  $\cos(\omega \ln t)$ and $\sin(\omega \ln\ln t)$, which are oscillating but not periodic.

Even for the homogeneous system, i.e. $f=0$, the current work still adds some contributions to the current literature.
	In \cite{Minea}, although the condition on $A$ is the same as in \ref{new2}, the function $F(y):=-Ay+G(y)$ is required to be analytic. The constructions of $q_k$'s are not explicit and  depend crucially on the analytic flows associated with the vector field $F(y)$. 	On contrary, this paper and \cite{CaH3} do not require the analyticity and, hence, cannot use the constructions in \cite{Minea}. 
	The work \cite{Shi2000}, when viewed in the ODE context, requires $A$ to be complex diagonalizable, which is more stringent than \ref{new2} above.
	Our constructions in \ref{new4} will be more concrete than \cite{Minea} and more general than \cite{Shi2000,CaH3}.

There are other asymptotic expansion theories for analytic ODEs such as Lyapunov's First Method,  see e.g.\cite[Chapter I, \S4]{BibikovBook} or \cite[Chapter V, \S3]{LefschetzBook}, and the Poincar\'e--Dulac normal form. 
The latter normal form theory has been developed much more by Bruno into the theory of power geometry,  see \cite{BrunoBook1989,BrunoBook2000,Bruno2004,Bruno2008c,Bruno2012,Bruno2018} and the references therein.
However, his classes of equations, results and methods are  very different from ours. 

The paper is organized as follows.
	Section \ref{notasec} presents the notation that we use throughout. 
	Section \ref{classtype} gives definitions of the asymptotic expansions that we study, see Definitions \ref{EEdef} and \ref{Lexpand}. For the latter, they involve specific classes of power functions which allow the bases to be logarithmic and iterated logarithmic functions, and the powers to be complex numbers. Therefore, they are more general and sophisticated than the previously studied expansions in \cite{Minea,Shi2000,CaH1,CaH3}.
	Section \ref{assumsec} imposes the main conditions in Assumption \ref{assumpA}  for the matrix $A$ and in Assumption \ref{assumpG} for the nonlinear function $G$. 
	In Section \ref{tools}, we establish the asymptotic approximations, as $t\to\infty$, for the solutions of the linearized equations of \eqref{sys}. Theorem \ref{approx1} treats the case of exponentially decaying forcing functions. The constructions of the approximating functions in Theorem \ref{approx1} are concrete, rather simple and close the original ideas of Foias and Saut. Theorem \ref{iterlog} deals with the case of power, logarithmically, and iterated logarithmically decaying forcing functions. The constructions in Theorem \ref{iterlog} are subtler than those in \cite{CaH3}, see operator $\mathcal Z_A$ in Definition \ref{defMRZ}. Besides their own merits, the results  are the key building blocks for the nonlinear problem later. 
	Section \ref{odesec} deals with the nonlinear systems. Theorem \ref{thmsmall} establishes the global existence and uniqueness for the solutions under the smallness conditions on the initial data and forcing functions. Solutions that decay to zero are also obtained. The main asymptotic estimate for large time is in Theorem \ref{thmdecay}. It provides a specific decaying rate for the solution corresponding to the asymptotic behavior of the forcing function. This  estimate will be needed in Sections \ref{eforce}--\ref{lforce}.
	Our main results for the systems in $\C^n$ are in Sections \ref{eforce}--\ref{lforce}, where we establish the asymptotic expansions for the decaying solutions. Theorem \ref{mainthm} is for exponentially decaying forcing functions,  Theorem \ref{mainthm2} is for power decaying forcing functions, and  Theorem \ref{mainthm3} is for logarithmically and iterated  logarithmically decaying forcing functions.
	The case of systems in $\R^n$ is dealt with in Section \ref{realsec}. The counterparts of the results in Sections \ref{eforce}--\ref{lforce} are Theorems \ref{newthmE} and \ref{newthmPL}. The proofs use the complexification technique but guarantee that the approximating functions still  stay in $\R^n$. These results are expressed even more clearly using only real-valued functions  in Theorems \ref{thmE3} and \ref{thmPL3}. 
	Appendix \ref{append} contains  elementary proofs of Lemma \ref{uniE} and Proposition \ref{unipres}.
	
	We end this Introduction with a note that the ideas and techniques in the current paper can be applied to problems in PDEs either straightforwardly or with some sophisticated adjustments.
	 
\section{Notation}\label{notasec}
We  use the following notation throughout the paper. 
\begin{itemize}
 \item $\N=\{1,2,3,\ldots\}$ denotes the set of natural numbers, and $\Z_+=\N\cup\{0\}$.
  
 \item Number $i=\sqrt{-1}$.

 \item For any vector $x\in\C^n$, we denote by $|x|$ its Euclidean norm, and by $x^{(k)}$ the $k$-tuple $(x,\ldots,x)$ for $k\in\N$, and  $x^{(0)}=1$. The real part, respectively, imaginary part,  of $x$ is denoted by $\Re x$, respectively, $\Im x$.
 
\item For an $m\times n$ matrix $M$ of complex numbers, its Euclidean norm in $\R^{mn}$ is denoted by $|M|$. 

\item We will use the convention 
\beqs
\sum_{k=1}^0 a_k=0\quad\text{and}\quad \prod_{k=1}^0 a_k=1.
\eeqs
 
 \item Let $f$ be a $\C^m$-valued function and $h$ be a non-negative function, both  are defined in a neighborhood of the origin in $\C^n$. 
 We write 
 $$f(x)=\bigo(h(x))\text{  as }x\to 0,$$
  if there are positive numbers  $r$ and $C$ such that $|f(x)|\le Ch(x)$ for all $x\in\C^n$ with $|x|<r$.   
 
 \item Let $f:[T_0,\infty)\to \C^n$ and $h:[T_0,\infty)\to[0,\infty)$ for some $T_0\in \R$. We write 
 $$f(t)=\bigo(h(t)), \text{ implicitly meaning as $t\to\infty$,} $$
 if there exist numbers $T\ge T_0$ and $C>0$ such that $|f(t)|\le Ch(t)$ for all $t\ge T$.

In the case $h(t)>0$ for all $t\ge T_0$, we write 
 $$f(t)=o(h(t)), \text{   if }\lim_{t\to\infty} \frac{|f(t)|}{h(t)}=0. $$
 
 \item Let $T_0\in\R$, functions $f,g:[T_0,\infty)\to\C^n$, and $h:[T_0,\infty)\to[0,\infty)$. 
We will conveniently write  
$$f(t)=g(t)+\bigo(h(t))\text{ to indicate } f(t)-g(t)=\bigo(h(t)).$$  
\end{itemize}

Let $\K=\R$ or $\C$.
If $m\in \N$ and $\mathcal L$ is an $m$-linear mapping (over $\K$) from $(\K^n)^m$ to $\K^n$, then the norm  $\|\mathcal L\|$ is defined by
\beqs
\|\mathcal L\|=\max\{ |\mathcal L(y_1,y_2,\ldots,y_m)|:y_j\in\K^n,|y_j|=1,\text{ for } 1\le j\le m\}.
\eeqs

Then $\|\mathcal L\|$ is a number in $[0,\infty)$, and
\beq\label{multiL}
|\mathcal L(y_1,y_2,\ldots,y_m)|\le \|\mathcal L\|\cdot |y_1|\cdot |y_2|\cdots |y_m|
\eeq
for any $y_1,y_2,\ldots,y_m\in\K^n$.

\begin{definition}\label{mset}
Let $S$ be a subset of $\C$. 
\begin{itemize}
\item We say $S$ preserves the addition if $x+y\in S$ for all $x,y\in S$.
\item We say $S$ preserves the unit increment if $x+1\in S$ for all $x\in S$.
\item The additive semigroup generated by $S$ is defined by 
\beqs
\langle S \rangle =\left\{\sum_{j=1}^N z_j:N\in\N,z_j\in S\text{ for }1\le j\le N\right\}.
\eeqs
\item The real part of $S$ is
$\Re S=\{\Re z:z\in S\}$.
\end{itemize}
\end{definition}

Regarding Definition \ref{mset}, it is obvious that $\langle S\rangle $ preserves the addition, and $\Re\langle S\rangle  =\langle\Re S\rangle$.

If $1\in S$, then $\langle S\rangle $ preserves the unit increment.

\section{Classes of functions and types of asymptotic expansions} \label{classtype}

\subsection{The complex power  function}
In this paper, we only deal with single-valued complex functions. To avoid any ambiguity we recall basic definitions and properties of elementary complex functions.

For $z\in\C$ and $t>0$, the exponential and power functions are defined by
\beq\label{ep}
\exp(z)=\sum_{k= 0}^\infty \frac{z^k}{k!}
\text{ and }
t^z=\exp(z\ln t).
\eeq

When $t=e=\exp(1)$ in \eqref{ep}, one has the usual identity
$e^z=\exp (z)$.

If $z=a+ib$ with $a,b\in\R$, then 
\beqs
t^z=t^a (\cos(b\ln t)+i\sin(b\ln t))\text{ and }|t^z|=t^a.
\eeqs

The standard properties of the power functions still hold, namely, 
\beqs
 t^{z_1}t^{z_2}=t^{z_1+z_2},\quad (t_1t_2)^z=t_1^z t_2^z,\quad (t^z)^m =t^{m z}= (t^m)^z,\quad \ddt (t^z)=zt^{z-1},
\eeqs
for any $t,t_1,t_2>0$, $z,z_1,z_2\in\C$, and $m\in\Z_+$.

\subsection{Iterated exponential and logarithmic functions}
\begin{definition}\label{ELdef} Define the iterated exponential and logarithmic functions as follows:
\begin{align*} 
&E_0(t)=t \text{ for } t\in\R,\text{ and } E_{m+1}(t)=e^{E_m(t)}  \text{ for } m\in \Z_+, \ t\in \R,\\
&\iln_{-1}(t)=e^t,\quad \iln_0(t)=t\text{ for } t\in\R,\text{ and }\\
& \iln_{m+1}(t)= \ln(\iln_m(t)) \text{ for } m\in \Z_+,\ t>E_m(0).
\end{align*}

For $k\in \Z_+$, define
\beqs 
\LL_k=(\iln_1,\iln_2,\ldots,\iln_k)\quad\text{and}\quad
\widehat \LL_k=(\iln_{-1},\iln_{0},\iln_{1},\ldots,\iln_{k}).
\eeqs
\end{definition}

Explicitly,
\beqs  
\widehat \LL_k(t)=(e^t,t,\ln t,\ln\ln t,\ldots,\iln_{k}(t)).
\eeqs 

The function $\LL_k$ is used to formulate the results in the previous work \cite{CaH3}. Instead, this paper will use $\widehat \LL_k$, which is more general than $\LL_k$.

For $m\in\Z_+$, note that 
\begin{align} 
&\text{$\iln_m(t)$ is positive and increasing  for $t>E_m(0)$, }\label{Linc}\\
&\iln_m(E_{m+1}(0))=1, \quad 
\lim_{t\to\infty} \iln_m(t)=\infty.\label{Lone}
\end{align}

It is also clear that 
\beq  \label{LLk}
\lim_{t\to\infty} \frac{\iln_k(t)^\lambda}{\iln_m(t)}=0\text{ for all }k>m\ge -1 \text{ and } \lambda\in\R.
\eeq 

For $m\in \N$, the derivative of $\iln_m(t)$ is
\beq\label{Lmderiv}
 \iln_m'(t)=\frac 1{t\prod_{k=1}^{m-1} \iln_k (t)}=\frac 1{\prod_{k=0}^{m-1} \iln_k (t)}. 
\eeq

By L'Hospital's rule and induction, one can verify that
\beq\label{Lshift}
\lim_{t\to\infty}\frac{\iln_m(T+t)}{\iln_m(t)}=1 \text{ for any $m\in\Z_+$ and $T\in \R$.}
\eeq

\subsection{Classes of functions}
Our results will cover a wide range of asymptotic expansions which  involve the following types of functions.

\begin{definition}\label{newclass} 
Let $X$ be a linear space over $\C$.
\begin{enumerate}[label=\tnum]
\item\label{dFE1} Define $ \mathcal F_E(X)$ to be the collection of functions $g:\R\to X$ of the form
\beq\label{gEform}
g(t)=\sum_{\lambda\in S} p_\lambda(t)e^{\lambda t} \text{ for }t\in\R,
\eeq
where $S$ is some finite subset of $\C$, and each $p_\lambda$ is a polynomial from $\R$ to $X$.

\item\label{dFE2} For $\mu\in\R$, define
 \beqs
 \mathcal F_E(\mu,X)=\Big \{\text{function } g(t)=\sum_{\lambda\in S} p_\lambda(t)e^{\lambda t}\in \mathcal F_E(X): \Re\lambda=\mu\text{ for all }\lambda\in S\Big\}.
 \eeqs
 \end{enumerate}
 \end{definition}
 
Clearly, $\mathcal F_E(X)$ is a linear space over $\C$ and $\mathcal F_E(\mu,X)$ is a subspace of $\mathcal F_E(X)$.  

 In particular, when $\mu=0$,
  \begin{align*}
 \mathcal F_E(0,X)
 &=\Big \{ \text{function } g(t)=\sum_{\omega\in \Omega} p_\omega(t)e^{i\omega t}:
 \text{ $\Omega$ is a finite subset of $\R$,} \\
&\qquad  \text{ each $p_\omega$ is an $X$-valued polynomial on $\R$} \Big\}.
 \end{align*}

It is clear that
 \beq\label{ghequiv}
 g\in  \mathcal F_E(\mu,X) \text{ if and only if } g(t)=h(t)e^{\mu t} \text{ for some } h\in  \mathcal F_E(0,X). 
 \eeq
 
Also, if $(X,\|\cdot\|_X)$ is a normed space and $g\in \mathcal F_E(\mu,X)$, then
 \beq\label{fEorder}
\| g(t)\|_X=\bigo(e^{(\mu+\delta)t}) \quad\text{ for all }\delta>0.
 \eeq

Regarding the uniqueness of the representation of $g(t)$ in \eqref{gEform}, we have the following lemma.

\begin{lemma}\label{uniE}
Let $X$ be a normed space over $\C$.
Suppose two functions $g$ and $h$ in $\mathcal F_E(\mu,X)$ satisfy 
\beq\label{gminush}
\lim_{t\to\infty} e^{-\mu t}(g(t)-h(t))=0,
\eeq
and are given by
\beq\label{gvsh}
g(t)=\sum_{\lambda\in S} p_\lambda(t)e^{\lambda t}\text{ and }
h(t)=\sum_{\lambda\in S} q_\lambda(t)e^{\lambda t},   \text{ for }t\in\R,
\eeq
where  $S$ is  a finite subset of $\C$, $\Re\lambda=\mu$ for all $\lambda\in S$, and $p_\lambda$'s and $q_\lambda$'s are polynomials from $\R$ to $X$.

Then $p_\lambda=q_\lambda$ for all $\lambda\in S$.
\end{lemma}

We give an elementary proof of Lemma \ref{uniE} in  Appendix  \ref{append}.

Concerning the assumption \eqref{gvsh}, the two functions $g$ and $h$  initially may not have the same set $S$ in their representations. For example,
 \beq\label{gh0}
g(t)=\sum_{\lambda\in S_g} p_\lambda(t)e^{\lambda t}\text{ and }
h(t)=\sum_{\lambda\in S_h} q_\lambda(t)e^{\lambda t},   \text{ for }t\in\R.
\eeq

By setting $S=S_g\cup S_h$,  and adding the zero functions to the sums for $g(t)$ and $h(t)$ when needed, we can rewrite \eqref{gh0} as \eqref{gvsh}.

\begin{proposition}\label{unipres}
Suppose $g\in \mathcal F_E(X)$ is given by \eqref{gEform} for some non-empty set $S$.
Then the polynomials $p_\lambda$'s are unique for $\lambda\in S$.
If, in addition, $p_\lambda\ne 0$ for all $\lambda\in S$, then the set $S$  is unique.
\end{proposition}

The proof of Proposition \ref{unipres} will be presented in Appendix  \ref{append}.

Next, we consider the power functions of several variables and complex exponents.
For 
\begin{align}
\label{azvec} 
z&=(z_{-1},z_0,z_1,\ldots,z_k)\in (0,\infty)^{k+2}
\text{ and } 
\alpha=(\alpha_{-1},\alpha_0,\alpha_1,\ldots,\alpha_k)\in \C^{k+2},
\end{align}
 define 
 \beq\label{vecpow}
 z^{\alpha}=\prod_{j=-1}^k z_j^{\alpha_j}.
 \eeq

\begin{definition}\label{specind}
For $\mu\in\R$, $m,k\in\Z$ with  $k\ge m\ge -1$, denote by $\mathcal E(m,k,\mu)$ the set of vectors $\alpha$ in \eqref{azvec} 
 such that
 \beqs
\Re(\alpha_j)=0 \text{ for $-1\le j<m$   and  } \Re(\alpha_m)=\mu.
\eeqs 
\end{definition}

In particular, when $m=-1$, $k\ge -1$, $\mu=0$, the set $\mathcal E(-1,k,0)$ is the collection of vectors $\alpha\in \C^{k+2}$ as in \eqref{azvec} with $\Re (\alpha_{-1})=0$.

Let $k\ge m\ge -1$, $\mu\in\R$, and $\alpha=(\alpha_{-1},\alpha_0,\ldots,\alpha_k)\in\mathcal E(m,k,\mu)$.
Consider $\widehat\LL_{k}(t)^\alpha$ for $t>E_k(0)$.

For $j<m$, the power $\alpha_j$ is imaginary, hence $|\iln_j(t)^{\alpha_j}|\le 1$.

For $j=m$, surely $|\iln_m(t)^{\alpha_m}|=L_m(t)^{\mu}$. 

For $j>m$, one has, thanks to \eqref{LLk},  $|\iln_j(t)^{\alpha_j}|=o(L_m(T_*+t)^s)$ for all $s>0$.

Therefore,
\beq\label{LLo}
\lim_{t\to\infty} \frac{\widehat\LL_{k}(t)^\alpha}{\iln_{m}(t)^{\mu+\delta}}=0 \quad\text{ for any }\delta>0.
\eeq

 \begin{definition}\label{Fclass}
 Let $\K$ be $\C$ or $\R$,  and  $X$ be a linear space over $\K$.

\begin{enumerate}[label=\tnum]
\item For $k\ge -1$, define $\mathscr P(k,X)$ to be the set of functions of the form 
\beq\label{pzdef} 
p(z)=\sum_{\alpha\in S}  z^{\alpha}\xi_{\alpha}\text{ for }z\in (0,\infty)^{k+2},
\eeq 
where $S$ is some finite subset of $\K^{k+2}$, and each $\xi_{\alpha}$ belongs to $X$.

\item Let $\K=\C$, $k\ge m\ge -1$ and $\mu\in\R$. 

Define $\mathscr P_{m}(k,\mu,X)$ to be set of functions of the form \eqref{pzdef},
where $S$ is a finite subset of $\mathcal E(m,k,\mu)$ and each $\xi_{\alpha}$ belongs to $X$.

Define
 \beqs
 \mathscr F_{m}(k,\mu,X)=\Big\{ p\circ \widehat{\LL}_{k}: p\in\mathscr P_{m}(k,\mu,X)\Big\}.
 \eeqs
 \end{enumerate}
 \end{definition}

Explicitly, if $f\in \mathscr F_{m}(k,\mu,X)$ then it is a function from $(E_k(0),\infty)$ to $X$ 
and
\beqs
f(t)=\sum_{\alpha\in S}  \widehat \LL_k(t)^\alpha \xi_\alpha\quad\text{ for } t>E_k(0),
\eeqs
where $S$ is a finite subset of $\mathcal E(m,k,\mu)$, and all $\xi_\alpha$'s belong to $X$.
 
Below are immediate observations about Definition \ref{Fclass}. 

\begin{enumerate}[label=(\alph*)]
 \item\label{Ca} $\mathscr P(k,X)$ contains all polynomials from $\R^{k+2}$ to $X$, in the sense that, if $p:\R^{k+2}\to X$ is a polynomial, then its restriction on $(0,\infty)^{k+2}$ belongs to $\mathscr P(k,X)$.
 
 \item\label{Cb} Each $\mathscr P(k,X)$ is a linear space over $\K$.

 \item\label{Cc} If $m>k\ge -1$, then $\mathscr P(k,X)$ can be embedded into $\mathscr P(m,X)$, see Remark (c) after Definition 2.7 in \cite{CaH3}.
 
 \item\label{Cd} One has
 \beq\label{qpequiv}
 q\in  \mathscr F_{m}(k,\mu,X) \text{ if and only if } q(t)=p(t) L_m(t)^{\mu} \text{ for some } p\in  \mathscr F_{m}(k,0,X). 
 \eeq
 
 \item \label{Ce} For any $k\ge m\ge 0$ and $\mu\in\R$, one has
\beq\label{Pm10}
\mathscr P_m(k,\mu,X)\subset \mathscr P_{-1}(k,0,X) . 
\eeq

 \end{enumerate}

\begin{lemma}\label{Pprop}
In this lemma $\K=\R$ or $\C$, and $X$, $Y$ and  $X_j$'s are linear spaces over $\K$.

The following statements hold true.
\begin{enumerate}[label={\tnum}]
 \item \label{P1} If $p_j\in\mathscr P(k,X_j)$ for $1\le j\le m$, where $m\ge 1$ and $L$ is an $m$-linear mapping from $\prod_{j=1}^m X_j$ to $X$, then
 $L(p_1,p_2,\ldots,p_m)\in\mathscr P(k,X)$. 
 \item \label{P2} If $p\in\mathscr P(k,X)$ and $L:X\to Y$ is a linear mapping, then $Lp\in\mathscr P(k,Y)$.
 \item \label{P3} If $p\in\mathscr P(k,\K^n)$ and $1\le j\le n$, then  the canonical projection $\pi_j p$, that maps $p$ to its $j$-th component, belongs to $\mathscr P(k,\K)$. 
 \item \label{P4} If $p\in \mathscr P(k,\K)$ and $q\in \mathscr P(k,X)$, then the product $p q\in \mathscr P(k,X)$.
 
 Consequently, if $p_j\in \mathscr P(k,\K)$ for $1\le j\le m$, then $p_1 p_2\ldots p_m\in \mathscr P(k,\K)$.
 
 \item \label{P5} If $p\in\mathscr P(k,\K^n)$ and $q$ is a polynomial from $\K^n$ to $X$, then the composition $q\circ p$ belongs to $\mathscr P(k,X)$. 
 \item \label{P6} In case $X$ is a normed space, if $p\in \mathscr P(k,X)$, then so is each partial derivative $\partial p(z)/\partial z_j$, for $z=(z_{-1},z_0,\ldots,z_k)\in(0,\infty)^{k+2}$ and $-1\le j\le k$.
\end{enumerate} 
\end{lemma}

Lemma \ref{Pprop} is stated and proved in \cite[Lemma 2.8]{CaH3} for $\K=\R$. However, it is equally true for $\K=\C$ with the same proof.

\subsection{Types of asymptotic expansions}

First, we give a definition for the asymptotic expansions which have the exponential functions as the main decaying modes.

\begin{definition}\label{EEdef} 
 Let $(X,\|\cdot\|_X)$ be a normed space over $\C$, and $g$ be a function from $(T,\infty)$ to $X$ for some $T\in\R$.

\begin{enumerate}[label=\tnum] 
 \item\label{EE1} Let $(\gamma_k)_{k=1}^\infty$ be a divergent, strictly increasing sequence of nonnegative numbers.
We say
  \beq \label{expan}
  g(t) \sim \sum_{k=1}^\infty g_k(t),   \text{ where $g_k\in\mathcal F_E(-\gamma_k,X)$   for $k\in\N$, }
  \eeq
 if for any $N\ge 1$, there exists $\mu>\gamma_N$ such that
  \beqs
  \left\|g(t)-\sum_{k=1}^N g_k(t)\right\|_X=\bigo(e^{-\mu t}).  
  \eeqs
 
 \item\label{EE2} Let $N\in\N$, and $(\gamma_k)_{k=1}^N$ be nonnegative and strictly increasing.
We say
  \beq\label{finex}
g(t) \sim \sum_{k=1}^N g_k(t), \text{ where $g_k\in\mathcal F_E(-\gamma_k,X)$  for $1\le k\le N$,  }
\eeq
if
\beqs
\left \|g(t) - \sum_{k=1}^N g_k(t)\right \|_X=\bigo(e^{-\mu t})\quad \text{ for any }\mu>0.
\eeqs
\end{enumerate}
\end{definition}
  
In case $X$ is a finite dimensional normed space, all norms on $X$ are equivalent. Hence, the above definitions of \eqref{expan} and \eqref{finex} are independent of the particular norm $\|\cdot\|_X$. 

Thanks to either \cite[Lemma 2.3]{HTi1} or Lemma  \ref{uniE},  see also  \cite[page 195]{Minea}, the asymptotic expansion \eqref{expan} for a function $g(t)$ is unique.

\begin{remark}\label{equivrmk}
The meaning of the asymptotic expansions \eqref{expan} and \eqref{finex}  can be seen more clearly when they are stated with the use of the equivalence form in \eqref{ghequiv} for $g_k(t)$. 
For example, an equivalence of \eqref{expan} is the following.

\textit{We say
  \beq \label{expan2}
  g(t) \sim \sum_{k=1}^\infty \widehat g_k(t) e^{-\gamma_k t}, \text{  where $\widehat g_k\in\mathcal F_E(0,X)$ for $k\in\N$,  }
  \eeq
if for any $N\ge 1$, there exists $\mu>\gamma_N$ such that
  \beq\label{gdrem2}
  \left\|g(t)-\sum_{k=1}^N \widehat g_k(t) e^{-\gamma_k t}\right\|_X=\bigo(e^{-\mu t}).  
  \eeq
}

Note that the function $\widehat g_k(t) $ in \eqref{gdrem2} does not contribute  any extra exponential decay to the decaying mode $e^{-\gamma_k t}$.
\end{remark}

Next, we define the asymptotic expansions in which the power or logarithmic or iterated logarithmic functions are the main decaying modes.

\begin{definition}\label{Lexpand}
Let  $(X,\|\cdot\|_X)$ be a normed space over $\C$. Suppose $g$ is a function from $(T,\infty)$ to $X$ for some $T\in\R$, and $m_*\in \Z_+$. 

\begin{enumerate}[label={\tnum}]
 \item\label{LE1} Let $(\gamma_k)_{k=1}^\infty$ be a divergent, strictly increasing sequence of positive numbers, and $(n_k)_{k=1}^\infty$ be a sequence in $\N\cap[m_*,\infty)$. 
We say
\beq\label{fiter}
g(t)\sim \sum_{k=1}^\infty g_k(t), \text{ where $g_k\in \mathscr F_{m_*}(n_k,-\gamma_k,X)$ for $k\in\N$, }
\eeq
if, for each $N\in\N$, there is some $\mu>\gamma_N$ such that
\beqs
\left\|g(t) - \sum_{k=1}^N g_k(t)\right\|_X=\bigo(\iln_{m_*}(t)^{-\mu}).
\eeqs

\item Let $N\in\N$, $(\gamma_k)_{k=1}^N$ be positive and strictly increasing, and  $n_*\in\N\cap[m_*,\infty)$.
We say
\beq\label{expan4}
g(t)\sim \sum_{k=1}^N g_k(t), \text{ where $g_k\in \mathscr F_{m_*}(n_*,-\gamma_k,X)$ for $1\le k\le N$, }
\eeq
if it holds for all $\mu>0$ that
\beqs 
\left\|g(t) - \sum_{k=1}^N g_k(t)\right\|_X=\bigo(\iln_{m_*}(t)^{-\mu}).
\eeqs
\end{enumerate}
\end{definition}

Similar to Remark \ref{equivrmk}, by using the equivalence \eqref{qpequiv}, we have the following equivalent form of \eqref{fiter}
\beq\label{expan3}
g(t)\sim \sum_{k=1}^\infty \widehat g_k(\widehat\LL_{n_k}(t))\iln_{m_*}(t)^{-\gamma_k}, \text{ where $\widehat g_k\in \mathscr P_{m_*}(n_k,0,X)$ for $k\in\N$. }
\eeq

For example, when $m_*=0$ the asymptotic expansion \eqref{expan3} reads as
\beqs
g(t)\sim \sum_{k=1}^\infty \widehat g_k(\widehat\LL_{n_k}(t)) t^{-\gamma_k},
\text{ where $\widehat g_k\in \mathscr P_{0}(n_k,0,X)$ for $k\in\N$. }
\eeqs

Same equivalence also applies to \eqref{expan4}.

\section{The main assumptions}\label{assumsec}

Let $n\in \N$ be fixed throughout the paper.
Consider the following system of nonlinear ODEs in $\C^n$: 
\beq \label{sys-eq}
y'=-Ay +G(y)+f(t),
\eeq 
where $A$ is an $n\times n$ constant matrix of complex numbers, $G$ is a vector field on $\C^n$, and  $f$ is a function from $(0,\infty)$ to $\C^n$. 
 
  \medskip
 \emph{The following Assumptions \ref{assumpA} and \ref{assumpG} will be imposed throughout Sections \ref{assumsec}--\ref{lforce}.} 

\begin{assumption}\label{assumpA}
All eigenvalues of the matrix $A$ have positive real parts.
 \end{assumption}
 
This assumption is as general as \cite{Minea}, and more general than \cite{Shi2000}. It is very often used to prove the asymptotic stability of an equilibrium.
 
\begin{assumption}\label{assumpG}  Function $G:\C^n\to\C^n$ has the the following properties. 
\begin{enumerate}[label=\tnum]
 \item\label{aG1}  $G$ is locally Lipschitz.
 \item\label{aG2} There exist functions $G_m:\C^n\to\C^n$,  for $m\ge 2$, each is a homogeneous polynomial of degree $m$, such that, for any $N\ge 2$, there exists $\delta >0$ so that
 \beq\label{Grem1}
 \left|G(x)-\sum_{m=2}^N G_m(x)\right|=\bigo(|x|^{N+\delta})\text{ as } x\to 0.
 \eeq
 \end{enumerate}
\end{assumption}

For the sake of brevity, we write Assumption \ref{assumpG}\ref{aG2} as
\beq \label{Gex}
G(x)\sim \sum_{m=2}^\infty G_m(x)\text{ as }x\to 0.
\eeq 

Obviously, any  $C^\infty$-function $G$ with $G(0)=0$ and derivative matrix $DG(0)=0$ satisfies Assumption \ref{assumpG}.  However,  the function $G$, in general,  is not required to analytic in a neighborhood of the origin.

We examine Assumptions \ref{assumpA} and \ref{assumpG} more now.

\subsection{The linear part}
Denote by $\Lambda_k$, for $1\le k\le n$, the eigenvalues of $A$ counting the multiplicities. The spectrum of $A$ is
$$\sigma(A)=\{\Lambda_k:1\le k\le n\}\subset \C.$$

Thanks to Assumption \ref{assumpA},  $\Re\Lambda_j>0$ for all $j$.
We order the set $\Re\sigma(A)$ by strictly increasing numbers $\lambda_j$'s, with $1\le j\le d$ for some $d\le n$.
Of course,
\beqs
0<\lambda_1\le\Re\Lambda_k\le \lambda_d \quad\text{ for } k=1,2,\ldots,n.
\eeqs

It follows that for any $\varep>0$,  there exists a positive constant $c_\varep$ such that
\beq\label{eAe}
\left|e^{-tA}\right|\le c_\varep e^{-(\lambda_1-\varep)t}\text{ for all } t\ge 0.
\eeq

In particular, there is $C_0>0$ such that
\beq\label{eA2}
\left |e^{-tA}\right|\le C_0 e^{-\lambda_1 t/2}\quad \text{ for all } t\ge 0.
\eeq

\subsection{The nonlinear part}
Consider condition \eqref{Gex}. For each $m\ge 2$,  there exists an $m$-linear mapping $\mathcal G_m$ from $(\C^n)^m$ to $\C^n$ such that
\beq\label{GG} G_m(x)=\mathcal G_m(x,x,\ldots,x)\text{ for } x\in\C^n.\eeq

By \eqref{multiL}, one has, for any $x_1,x_2,\ldots,x_m\in\C^n$, that
\beq\label{multineq}
|\mathcal G_m(x_1,x_2,\ldots,x_m)|\le \|\mathcal G_m\|\cdot |x_1|\cdot |x_2|\cdots |x_m|.
\eeq

In particular,
\beq\label{multy}
|G_m(x)|\le \|\mathcal G_m\|\cdot |x|^m\quad\text{ for all } x\in\C^n.
\eeq

It follows \eqref{Grem1}, when $N=2$, and \eqref{multy}, for $m=2$, that $|G(x)|=\bigo(|x|^2)$ as $x\to 0$. Thus, there exist positive numbers $c_*$ and $r_*$ such that 
\beq \label{Gyy}
|G(x)|\le c_*|x|^2 \quad\forall x\in\C^n \text{ with } |x|<r_*.
\eeq 

\section{Asymptotic approximations for solutions of the linear system}\label{tools}

This section is focused on the linearization of system \eqref{sys-eq} around the origin.
The linearized system contains a forcing function which consists of two parts: the coherently decaying part $f$ and the faster decaying part $g$. The task is to approximate its solutions by the coherently decaying solutions of the same linear system but with $f$ only.

\subsection{Forcing function with exponential decay}
We recall an useful, elementary integral formula. 
If $B$ is an invertible $k\times k$ matrix of complex numbers, and $p:\R\to\C^k$ is a $\C^k$-valued polynomial, then integration by parts repeatedly yields
\beq \label{intep}
\int e^{tB}p(t)\d t=\sum_{k=0}^{\deg(p)} (-1)^k e^{tB} B^{-k-1}\frac{\d^k p(t)}{\d t^k}+C,
\eeq 
where $C$ is any constant vector in $\C^k$.

 \begin{theorem}\label{approx1}
  Given $\mu>0$, $f\in\mathcal F_E(-\mu,\C^n)$ and a function $g\in C([T,\infty),\C^n)$, for some $T\ge 0$, that satisfies 
  \beq\label{gO} 
  g(t)=\bigo(e^{-(\mu+\delta)t}) \text{ for some }\delta>0.
  \eeq
  
  Assume  $y\in C([T,\infty),\C^n)$ is a solution of
  \beq\label{ayeq}
  y'(t)+Ay(t)=f(t)+g(t),\quad\text{ for } t>T,
  \eeq
  and it holds for any $\lambda\in\Re\sigma(A)$ with $\lambda<\mu$ and any number $m\in\N$ that
  \beq\label{ylimcond}
  \lim_{t\to\infty} t^m e^{\lambda t}|y(t)|=0.
  \eeq
  
  Then there exists a function  $z\in \mathcal F_E(-\mu,\C^n)$ and a number $\varep>0$ such that
  \beq\label{zeq}
  z'(t)+Az(t)=f(t) \quad\text{ for } t\in \R,
  \eeq
  and 
  \beq \label{ymz}
  |y(t)-z(t)|=\bigo(e^{-(\mu+\varep)t}).
  \eeq 
 \end{theorem}
\begin{proof}
Note from equation \eqref{ayeq} and the stated assumptions that, in fact, $y\in C^1([T,\infty),\C^n)$. 

Observe that the function $t\mapsto f(T+t)$ belongs to $\mathcal F_E(-\mu,\C^n)$ and the  function $t\mapsto g(T+t)$ is of $\bigo(e^{-(\mu+\delta)t}) $.
Moreover, as a consequence of \eqref{ylimcond},
\beq\label{ylimp}
  \lim_{t\to\infty} p(t)e^{\lambda_j t}|y(t)|=0\text{ for any polynomial $p:\R\to \R$.}
  \eeq

Let $m\in\N$ and $\tau=T+t$.  We write 
\beq\label{yshift}
 t^m e^{\lambda_j t}|y(T+t)|= e^{-\lambda_j T}(\tau-T)^m e^{\lambda_j \tau}|y(\tau)|.
\eeq

Thanks to \eqref{ylimp}, the right-hand side of \eqref{yshift} goes to zero as $\tau\to\infty$. Thus, 
$$\lim_{t\to\infty} t^m e^{\lambda_j t}|y(T+t)| = 0.$$ 

Therefore, we can translate the time variable and assume that $T=0$.

By projecting equation \eqref{ayeq} to the invariant subspaces  corresponding to the Jordan normal form of the matrix $A$, we can reduce the problem to the following
\beqs
y'(t)+By(t)=f(t)+g(t), \text{ for }t>0,
\eeqs
where $y(t)\in\C^{n'}$, $B=B_j$ is a Jordan matrix of size $n'\times n'$, for some $n'\le n$, corresponding to an eigenvalue $\Lambda_j$ of $A$, $f\in\mathcal F_E(-\mu,\C^{n'})$ and $g\in C([0,\infty),\C^{n'})$ satisfies \eqref{gO}.

Assume
\beq \label{ftsum}
f(t)=\sum_{\Re\lambda=\mu} f_\lambda(t)\text{ with } f_\lambda(t)=p_\lambda(t)e^{-\lambda t},
\eeq 
where each $p_\lambda$ is a polynomial from $\R$ to $\C^{n'}$.
In \eqref{ftsum} and throughout this proof, the sum $\sum_{\Re\lambda=\mu}$ is understood to be over finitely many $\lambda$'s.

It is known that 
\beq\label{etB}
e^{-tB}=e^{-\Lambda_j t}p(t),
\eeq
where $p(t)$ is an $n'\times n'$ matrix-valued polynomial in $t\in\R$.
Consequently, for any $\varep>0$, then, there exist positive constants $\widehat C_\varep$ and $\widetilde C_\varep$ such that one has, for all $t\ge 0$,
\begin{align}
\label{eBp}
|e^{tB}|&\le \widehat C_\varep e^{(\Re\Lambda_j+\varep)t},\\
\label{eBm}
 |e^{-tB}|&\le \widetilde C_\varep e^{-(\Re\Lambda_j-\varep)t}.
\end{align}

By \eqref{gO} and the continuity of $g$ on $[0,\infty)$, we can assume that  there is a constant $C_1>0$ such that
\beq\label{gexb}
|g(t)|\le C_1e^{-(\mu+\delta)t} \text{ for all $t\ge 0$.}
\eeq

\medskip
\noindent\textbf{Case 1: $\Re\Lambda_j>\mu$.} 
For $t\in\R$, let 
\beq\label{zdef1}
z(t)=\int_{-\infty}^t e^{-(t-\tau)B}f(\tau)\d\tau. 
\eeq

Then $z(t)$ solves equation \eqref{zeq} with $A:=B$.
Clearly, the difference $u\eqdef y-z$ satisfies the linear equation
\beqs
u'(t)+Bu(t)=g(t),\text{ for } t>0.
\eeqs

Hence, we have
\beq\label{yz}
y(t)-z(t)=u(t)=e^{-tB}u(0)+\int_0^t e^{-(t-\tau)B}g(\tau)\d\tau.
\eeq

Fix a number $\varep$ such that
\beqs
0<\varep<\min\left\{\delta,\frac{\Re \Lambda_j-\mu}{2}\right\}.
\eeqs

Note with this choice that $\Re \Lambda_j-\varep > \mu+\varep$. 
Using \eqref{eBm} to estimate the first exponential term on the right-hand side of \eqref{yz} gives 
\beq\label{eBu}
|e^{-tB}u(0)|\le \widetilde C_\varep e^{-(\Re\Lambda_j-\varep)t}|u(0)|\le \widetilde C_\varep e^{-(\mu+\varep)t}|u(0)|.
\eeq

Similarly, thanks to \eqref{eBm} and \eqref{gexb}, the integral in \eqref{yz}  is bounded by
\beqs
\left|\int_0^t e^{-(t-\tau)B}g(\tau)\d\tau\right|
\le \widetilde C_\varep C_1 \int_0^t e^{-(\mu+\varep)(t-\tau)} e^{-(\mu+\delta)\tau}\d\tau
=\widetilde C_\varep C_1  \frac{e^{-(\mu+\varep)t}}{\delta-\varep} (1-e^{-(\delta-\varep)t}).
\eeqs

Because $\delta>\varep$, we deduce
\beq\label{ieBg}
\left|\int_0^t e^{-(t-\tau)B}g(\tau)\d\tau\right|
\le \frac{\widetilde C_\varep C_1}{\delta-\varep} e^{-(\mu+\varep)t}.
\eeq

Thus, we obtain \eqref{ymz} from formula \eqref{yz} and estimates \eqref{eBu}, \eqref{ieBg}. It remains to prove $z\in \mathcal F_E(-\mu)$.
We rewrite \eqref{zdef1} as
\beqs
z(t)=\sum_{\Re\lambda=\mu} z_\lambda(t),
\eeqs 
 where 
\beqs
z_\lambda(t)
=\int_{-\infty}^t e^{-(t-\tau)B}f_\lambda(\tau)\d\tau
=e^{-t B} \int_{-\infty}^t e^{\tau(B-\lambda I_{n'})}p_\lambda(\tau)\d\tau. 
\eeqs

Applying formula \eqref{intep}, we can compute  
\beq\label{zeq1}
\begin{aligned}
z_\lambda(t)=e^{-t B}
&\left\{\sum_{k=0}^{\deg(p_\lambda)} (-1)^k e^{t(B-\lambda I_{n'})} (B-\lambda I_{n'})^{-k-1} \frac{\d^k p_\lambda(t)}{\d t^k}\right.\\
&\quad \left.-\lim_{s\to -\infty} \sum_{k=0}^{\deg(p_\lambda)} (-1)^k e^{s(B-\lambda I_{n'})} (B-\lambda I_{n'})^{-k-1} \frac{\d^k p_\lambda(s)}{\d s^k}\right\} .
\end{aligned}
\eeq 

Because $B-\lambda I_{n'}$ has the sole eigenvalue $\Lambda_j-\lambda$, and
$$\Re(\Lambda_j-\lambda)=\Re\Lambda_j-\mu>0,$$
 the norm $|e^{s(B-\lambda I_{n'})}|$ decays exponentially as $s\to-\infty$.
Meanwhile, each $\d^k p_\lambda(s)/\d s^k$ is a polynomial. Therefore, the limit as $s\to-\infty$ in \eqref{zeq1} is zero, and, consequently, 
\beq\label{zsum1}
z_\lambda(t)=e^{-\lambda t}\sum_{k=0}^{\deg(p_\lambda)} (-1)^k (B-\lambda I_{n'})^{-k-1}\frac{\d^k p_\lambda(t)}{\d t^k}.
\eeq

Thanks to the facts  $\Re\lambda=\mu$ and each $\d^k p_\lambda(t)/\d t^k$ in \eqref{zsum1}  is a $\C^{n'}$-valued polynomial, one obtains $z_\lambda\in \mathcal F_E(-\mu,\C^{n'})$ for each $\lambda$. Hence, 
$z\in \mathcal F_E(-\mu,\C^{n'})$.

\medskip
\noindent\textbf{Case 2: $\Re\Lambda_j=\mu$.} 
Denote $y_0=y(0)$.
By the variation of constants formula, we can write solution $y(t)$ as 
\beq\label{yy1}
y(t)=e^{-tB}y_0+\int_0^t e^{-(t-\tau)B}f(\tau)\d\tau+ e^{-tB}\int_0^\infty  e^{\tau B}g(\tau) \d\tau-\int_t^\infty  e^{-(t-\tau) B}g(\tau))\d\tau. 
\eeq

Let $\varep=\delta/2$. By using inequalities  \eqref{eBp} and \eqref{gexb}, we have 
\beqs
|e^{t B}g(t)|\le \widehat C_\varep C_1 e^{(\mu+\delta/2)t} e^{-(\mu+\delta)t}=\widehat C_\varep C_1 e^{-\delta/2 t} \quad\forall t\ge 0.
\eeqs

Thus, by defining $Y_0=\int_0^\infty  e^{\tau B}g(\tau) \d\tau$, one has  $Y_0$  is a vector in $\C^{n'}$.

For $t\in\R$, let 
\beq \label{zdef2}
z(t)=e^{-tB } (y_0+Y_0)+\int_0^t e^{-(t-\tau)B}f(\tau)\d\tau=e^{-tB } \xi+J(t),
\eeq 
where $\xi=y_0+Y_0$ and $J(t)=\int_0^t e^{-(t-\tau)B}f(\tau)\d\tau$.

We rewrite $y(t)$ from \eqref{yy1} as
\beq
y(t)=z(t)-\int_t^\infty  e^{(\tau-t) B}g(\tau))\d\tau ,\label{yy}
\eeq

Clearly, the function $z(t)$ satisfies equation \eqref{zeq} with $A:=B$. 
For $t>0$, one has, thanks to \eqref{yy} and inequalities \eqref{eBp}, \eqref{gexb}, 
\beqs
 | y(t)-z(t) |=\left |\int_t^\infty  e^{(\tau-t) B}g(\tau))\d\tau\right|
\le \widehat C_\varep C_1 \int_t^\infty e^{(\mu+\delta/2)(\tau-t)} e^{-(\mu+\delta)\tau}\d\tau
=\frac{2\widehat C_\varep C_1}{\delta}e^{-(\mu+\delta)t} .
\eeqs

Therefore, we obtain estimate \eqref{ymz}. 

We  prove  $z\in \mathcal F_E(-\mu,\C^{n'})$ now.
Thanks to formula \eqref{etB} and the fact  $\Re\Lambda_j=\mu$, one has 
\beq \label{term1}
\text{  the term  $e^{-tB } \xi$  in \eqref{zdef2} belongs to $\mathcal F_E(-\mu,\C^{n'})$.}
\eeq 

It remains to calculate the remaining term $J(t)$ in \eqref{zdef2}.
With $f(t)$ given in \eqref{ftsum}, we have
\beq\label{Jdef} 
J(t)=\sum_{\Re\lambda=\mu} J_\lambda(t),\text{ where } 
J_\lambda(t)=e^{-t B} \int_0^t e^{(B-\lambda I_{n'}) \tau}p_\lambda(\tau)\d\tau.
\eeq 

For each $J_\lambda(t)$, we consider the two cases $\lambda\neq \Lambda_j$ and $\lambda=\Lambda_j$ separately.

\medskip
\emph{Case 2a: $\lambda\neq \Lambda_j$.} Then $B-\lambda I_{n'}$ is invertible, and applying the integral formula \eqref{intep}  gives
\beq\label{zeq3}
\begin{aligned}
J_\lambda(t)=e^{-t B}&\left\{\sum_{k=0}^{\deg(p_\lambda)} (-1)^k e^{t(B-\lambda I_{n'})} (B-\lambda I_{n'})^{-k-1} \frac{\d^k p_\lambda(t)}{\d t^k}\right.\\
&\quad\left.-\sum_{k=0}^{\deg(p_\lambda)} (-1)^k (B-\lambda I_{n'})^{-k-1}\frac{\d^k p_\lambda(0)}{\d t^k}\right\} .
\end{aligned}
\eeq 

It follows that
\beq\label{zeq2}
\begin{aligned}
J_\lambda(t)&=e^{-\lambda t}\sum_{k=0}^{\deg(p_\lambda)} (-1)^k  (B-\lambda I_{n'})^{-k-1}\frac{\d^k p_\lambda(t)}{\d t^k}\\
&\quad  - e^{-t B}  \sum_{k=0}^{\deg(p_\lambda)} (-1)^k (B-\lambda I_{n'})^{-k-1}\frac{\d^k p_\lambda(0)}{\d t^k}.
\end{aligned}
\eeq 

The first term (including the first sum) on the right-hand side of  \eqref{zeq2} is the right-hand side of \eqref{zsum1}, hence, it belongs to $\mathcal F_E(-\mu,\C^{n'})$.
Same as \eqref{term1}, the second term  (including the second sum)  on the right-hand side of  \eqref{zeq2}  also belongs to $\mathcal F_E(-\mu,\C^{n'})$.
Therefore, 
\beq\label{J1} J_\lambda\in \mathcal F_E(-\mu,\C^{n'}).
\eeq 

\medskip
\emph{Case 2b: $\lambda=\Lambda_j$.} 
Note that $B-\lambda I_{n'}$ is a Jordan matrix having  zero as its only eigenvalue. 
Then $e^{(B-\lambda I_{n'}) \tau}$ is just a matrix-valued  polynomial in $\tau$. 
Hence, $\widetilde J_\lambda(t)\eqdef \int_0^t e^{(B-\lambda I_{n'}) \tau}p_\lambda(\tau)\d\tau$ is a polynomial in $t$. Combining this with equation  $J_\lambda(t)=e^{-tB}\widetilde J_\lambda(t)$ from \eqref{Jdef}, 
 formula \eqref{etB} for $e^{-tB}$, and the fact $\Re\Lambda_j=\mu$, we obtain \eqref{J1} again.

\medskip
By \eqref{zdef2}, \eqref{term1}, \eqref{Jdef},  and property \eqref{J1} for both cases 2a and 2b, we obtain $z\in \mathcal F_E(-\mu,\C^{n'})$.

\medskip
\noindent\textbf{Case 3: $\Re\Lambda_j<\mu$.} For any $s,t>T$, we have
\beq\label{yc3}
y(t)=e^{-(t-s)B}y(s)-e^{-tB}\int_t^s  e^{\tau B}(f(\tau)+g(\tau)) \d\tau.
\eeq 

Thanks to \eqref{etB} and condition  \eqref{ylimcond} applied to $\lambda=\Re\Lambda_j$, which belongs to $\Re\sigma(A)$ and is less than $\mu$, one has  $e^{sB}y(s)\to0$ as $s\to\infty$.
Then letting $s\to\infty$ in \eqref{yc3} gives
\beqs
y(t)=-e^{-tB}\int_t^\infty  e^{\tau B}(f(\tau)+g(\tau)) \d\tau.
\eeqs

For $t\in\R$, let 
\beq\label{zdef3}
z(t)=-e^{-tB}\int_t^\infty  e^{\tau B}f(\tau) \d\tau.
\eeq

It is obvious that $z(t)$ satisfies \eqref{zeq} with $A:=B$. 
For the difference between $y(t)$ and $z(t)$, we estimate
\beqs
|y(t)-z(t)|=\left | \int_t^\infty  e^{-(t-\tau) B}g(\tau) \d\tau\right|
\le \int_t^\infty | e^{(\tau-t) B}| |g(\tau)| \d\tau.
\eeqs

Let $\varep=\mu-\Re\Lambda_j>0$.
We apply inequality \eqref{eBp} to estimate  $| e^{(\tau-t) B}|$ and utilize estimate \eqref{gexb} of $|g(\tau)|$. It results in
\beqs
|y(t)-z(t)|\le \widehat C_\varep C_1 \int_t^\infty e^{\mu(\tau-t)}e^{-(\mu+\delta)\tau} \d\tau 
= \frac{\widehat C_\varep C_1}{\delta}e^{-(\mu+\delta)t} ,
\eeqs
which proves \eqref{ymz}.

We compute $z(t)$ now.
Again, rewrite  \eqref{zdef3} as
$z(t)=\sum_{\Re\lambda=\mu} z_\lambda(t)$,
where
\beqs
z_\lambda(t)=-e^{-tB}\int_t^\infty  e^{\tau B}f_\lambda(\tau) \d\tau
=-e^{-t B} \int_t^{\infty} e^{\tau(B-\lambda I_{n'})}p_\lambda(\tau)\d\tau. 
\eeqs

Since $B-\lambda I_{n'}$ is invertible and, similar to \eqref{zeq1} and \eqref{zeq3}, we have
\beq\label{zdef4}
\begin{aligned}
z_\lambda(t)=-e^{-t B}
& \left\{  \lim_{s\to \infty}\left( \sum_{k=0}^{\deg(p_\lambda)} (-1)^k e^{s(B-\lambda I_{n'})} (B-\lambda I_{n'})^{-k-1} \frac{\d^k p_\lambda(s)}{\d s^k}\right)\right.\\
& \quad \left.-\sum_{k=0}^{\deg(p_\lambda)} (-1)^k e^{t(B-\lambda I_{n'})} (B-\lambda I_{n'})^{-k-1}\frac{\d^k p_\lambda(t)}{\d t^k}
\right\} .
\end{aligned}
\eeq 

In this case, $|e^{s(B-\lambda I_{n'})}|$ decays exponentially as $s\to\infty$, while $\d^k p_\lambda(s)/\d s^k$ is a polynomial in $s$.  As a result, the limit as $s\to\infty$ in \eqref{zdef4} is zero. Then simplifying  the remaining part of $z(t)$ in \eqref{zdef4} yields the same formula \eqref{zsum1}
 for $z_\lambda(t)$.
This formula, again, implies $z_\lambda\in \mathcal F_E(-\mu,\C^{n'})$ and, consequently, $z\in \mathcal F_E(-\mu,\C^{n'})$.
\end{proof}
 \subsection{Forcing function with power or logarithmic or iterated logarithmic decay}

The following linear transformations will play crucial roles in our presentation.
 
\begin{definition}\label{defMRZ}
Given an integer $k\ge -1$, let $p\in \mathscr P(k,\C^n)$ be given by \eqref{pzdef} with   $\K=\C$, $X=\C^n$, and $z$ and $\alpha$ as in \eqref{azvec}. 

Define, for $j=-1,0,\ldots,k$, the function $\mathcal M_jp:(0,\infty)^{k+2}\to \C^n$ by 
\beq\label{MM}
(\mathcal M_jp)(z)=\sum_{\alpha\in S} \alpha_j z^\alpha \xi_\alpha.
\eeq 

In the case $k\ge 0$, define the function $ \mathcal R  p:(0,\infty)^{k+2}\to \C^n$ by 
 \beq\label{chiz}
 (\mathcal R p)(z)=
  \sum_{j=0}^k z_0^{-1}z_1^{-1}\ldots z_{j}^{-1}(\mathcal M_j p)(z).
 \eeq 

In the case $p\in \mathscr P_{-1}(k,0,\C^n)$, define the function $\mathcal Z_Ap:(0,\infty)^{k+2}\to \C^n$ by 
 \beq\label{ZAp}
 (\mathcal Z_Ap)(z)=\sum_{\alpha\in S}  z^{\alpha}(A+\alpha_{-1}I_n)^{-1} \xi_{\alpha}.
 \eeq
\end{definition}

In particular,
\beqs
\mathcal M_{-1}p(z)=\sum_{\alpha\in S} \alpha_{-1} z^\alpha \xi_\alpha
\quad \text{and}\quad 
\mathcal M_0p(z)=\sum_{\alpha\in S} \alpha_0 z^\alpha \xi_\alpha .
\eeqs
 
An equivalent definition of $(\mathcal Rp)(z)$ in \eqref{chiz} is
\beq
 (\mathcal R p)(z)=\frac{\partial p(z)}{\partial z_0}+
  \sum_{j=1}^k z_0^{-1}z_1^{-1}\ldots z_{j-1}^{-1}\frac{\partial p(z)}{\partial z_j}.
\eeq 

Note in \eqref{ZAp} that  $\Re(\alpha_{-1})=0$ for all $\alpha\in S$.
The eigenvalues of $A+\alpha_{-1}I_n$ are $\Lambda_j+\alpha_{-1}$ for $1\le j\le n$. 
With $\Re(\Lambda_j+\alpha_{-1})=\Re \Lambda_j>0$, we have  $A+\alpha_{-1}I_n$ is invertible and definition \eqref{ZAp} is valid.

From \eqref{ZAp}, one has $\mathcal Z_Ap\in \mathscr P_{-1}(k,0,\C^n)$. Clearly,
\beq\label{ZAM}
 (A+\mathcal M_{-1})(\mathcal Z_Ap)=\mathcal Z_A((A+\mathcal M_{-1})p)=p\quad \forall p\in \mathscr P_{-1}(k,0,\C^n).
 \eeq
 
 By the mappings $p\mapsto \mathcal M_j p$, $p\mapsto \mathcal Rp$, and $p\mapsto \mathcal Z_A p$, one can define
 linear operator $\mathcal M_j$, for $-1\le j\le k$, on $\mathscr P(k,\C^n)$,
  linear operator $\mathcal R$ on $\mathscr P(k,\C^n)$ for $k\ge 0$,
 and linear operator  $\mathcal Z_A$ on $\mathscr P_{-1}(k,0,\C^n)$ for $k\ge -1$.

The powers $\alpha$'s in \eqref{MM} for $\mathcal M_jp(z)$, and  in \eqref{ZAp} for $(\mathcal Z_Ap)(z)$ are the same as those that appear in \eqref{pzdef} for $p(z)$. 
Consequently, we have the following \ properties. 
\begin{enumerate}[label=\rnum]
\item \label{R0} For $k\ge m\ge 0$ and $\mu\in\R$,   if $p$ is in $\mathscr P_{m}(k,\mu,\C^n)$,  then all $\mathcal M_jp$'s, for $-1\le j\le k$,  and $\mathcal Z_Ap$  are also in $\mathscr P_{m}(k,\mu,\C^n)$. Here,  $\mathcal Z_Ap$  is validly defined thanks to the inclusion \eqref{Pm10}.

\item\label{R1} 
 $\mathcal R p(z)$ has the same powers of  $z_{-1}$ as $p(z)$.
 
\item \label{R2}   
If $p\in\mathscr P_0(k,\mu,\C^n)$, then $\mathcal R p\in\mathscr P_0(k,\mu-1,\C^n)$.
\end{enumerate}

We recall a useful inequality from \cite{CaH3}.

\begin{lemma}[{\cite[Lemma 2.5]{CaH3}}]\label{plnlem}
Let $m\in\Z_+$ and $\lambda>0$, $\gamma>0$  be given. For any number $T_*>E_m(0)$, there exists a number $C>0$ such that
\beq\label{iine2}
 \int_0^t e^{-\gamma (t-\tau)}\iln_m(T_*+\tau)^{-\lambda}\d\tau
 \le C \iln_m(T_*+t)^{-\lambda} \quad\text{for all }t\ge 0.
\eeq
\end{lemma}

The next lemma contains the asymptotic approximations of the integrals that appear in the variation of constants formula for the solutions of many linear ODE systems.

\begin{lemma}\label{newplem} 
Let $m\in Z_+$, $\mu>0$, $N\ge m$, $p\in\mathscr P_{m}(N,-\mu,\C^n)$ and $T_*>E_N(0)$.

If $m=0$, then 
 \beq\label{log1}
 \int_0^t e^{-(t-\tau)A} p(\widehat \LL_N(T_*+\tau)) \d\tau=(\mathcal Z_A p)(\widehat \LL_N(T_*+t)) 
 +\bigo((T_*+t)^{-\mu-\gamma})
 \eeq
 for any $\gamma\in(0,1)$.

If $m\ge1$, then
 \beq\label{log2}
 \int_0^t e^{-(t-\tau)A} p(\widehat \LL_N(T_*+\tau)) \d\tau=(\mathcal Z_A p)(\widehat \LL_N(T_*+t)) 
 +\bigo((T_*+t)^{-1}),
 \eeq
 and, consequently,
 \beq\label{log3}
 \int_0^t e^{-(t-\tau)A} p(\widehat \LL_N(T_*+\tau)) \d\tau=(\mathcal Z_A p)(\widehat \LL_N(T_*+t)) 
 +\bigo(L_m(T_*+t)^{-\gamma}) 
 \eeq
 for all $\gamma>0$.
 \end{lemma}
\begin{proof}
It suffices to prove \eqref{log1} and \eqref{log2} for 
$p(z)= z^{ \alpha}\xi$, with 
$$z=(z_{-1},z_0,\ldots,z_k)\in(0,\infty)^{k+2},\
\alpha=(\alpha_{-1},\alpha_0,\ldots,\alpha_k)\in\mathcal E(m,N,-\mu),
$$
and some $\xi\in \C^n$.

Since $m\ge 0$, we have $\alpha_{-1}=i\beta$ for some $\beta\in\R$.
By defining  $F(t)=\prod_{j=0}^N \iln_j(T_*+t)^{\alpha_j} \xi$, we can write 
 \beqs
 p(\widehat \LL_N(T_*+t)) =e^{i\beta (T_*+t)}F(t).
 \eeqs 
 
Denote $I(t)=\int_0^t e^{-(t-\tau)A} p(\widehat \LL_N(T_*+\tau)) \d\tau$. Then
 \begin{align*}
 I(t)= \int_0^t e^{i\beta(T_*+ \tau)}e^{-(t-\tau)A}F(\tau)   \d\tau 
= e^{i\beta (T_*+t)}\int_0^t e^{-(t-\tau)(A+i\beta I_n)}F(\tau)    \d\tau.
 \end{align*}

Note that the function $\tau\mapsto e^{-(t-\tau)(A+i\beta I_n)}$ has  an anti-derivative $(A+i\beta I_n)^{-1} e^{-(t-\tau)(A+i\beta I_n)}$.
Then using integration by parts gives
  \begin{align*}
  I(t) 
&=  e^{i\beta (T_*+t)}\left\{(A+i\beta I_n)^{-1} e^{-(t-\tau)(A+i\beta I_n)}F(\tau)   \Big|_{\tau=0}^{\tau=t}\right.\\
&\qquad\qquad\qquad \left.- \int_0^t (A+i\beta I_n)^{-1}e^{-(t-\tau)(A+i\beta I_n)} \frac{\d F(\tau)}{\d \tau}  \d\tau \right\}\\
&=e^{i\beta (T_*+t)}(A+i\beta I_n)^{-1}F(t) -e^{i\beta T_*}(A+i\beta I_n)^{-1} e^{-tA}F(0)  -J(t),
 \end{align*} 
where
\beqs
J(t)=\int_0^t e^{i\beta(T_*+\tau)}(A+i\beta I_n)^{-1}e^{-(t-\tau)A} \frac{\d F(\tau)}{\d \tau}  \d\tau.
\eeqs

Hence,
\beq\label{Iint}
  I(t) = (A+i\beta I_n)^{-1}p(\widehat \LL_N(T_*+t)) -(A+i\beta I_n)^{-1}e^{-tA} p(\widehat \LL_N(T_*))   -J(t).
\eeq

We consider each term on the right-hand side of \eqref{Iint}.
For the first term, it is obvious that
\beq\label{Iz}
(A+i\beta I_n)^{-1}p(\widehat \LL_N(T_*+t)) =(\mathcal Z_Ap)(\widehat \LL_N(T_*+t)) .
\eeq

For the second term, one has, thanks to \eqref{eA2},  
\beq\label{Iini}
(A+i\beta I_n)^{-1}e^{-tA} p(\widehat \LL_N(T_*))=\bigo(e^{-\lambda_1 t/2}).
\eeq

For the third term,  we estimate $J(t)$ by applying inequality \eqref{eA2} again to have 
\beq\label{JJ}
|J(t)|\le C_1\int_0^t e^{-\lambda_1(t-\tau)/2} \left|\frac{\d F(\tau)}{\d \tau} \right| \d\tau,\text{ where }
C_1=C_0|(A+i\beta I_n)^{-1}|.
\eeq

Taking the derivative of $F(t)$, by using the product rule and \eqref{Lmderiv}, yields
 \begin{align*}
 \frac{\d F(t)}{\d t}
 &=  \alpha_0 (T_*+t)^{-1} \prod_{j=0}^N (\iln_j(T_*+\tau))^{\alpha_j}\xi\\
&\quad  + \sum_{k=1}^N \left[ \alpha_k \iln_k(T_*+t)^{-1} \prod_{j=0}^N \iln_j(T_*+\tau)^{\alpha_j} \left((T_*+t)\prod_{\ell=1}^{k-1}\iln_\ell(T_*+t) \right)^{-1}\right] \xi\\
&=(T_*+t)^{-1}\left\{ \alpha_0
+ \sum_{k=1}^N \left[\alpha_k  \prod_{\ell=1}^k (\iln_\ell(T_*+t))^{-1}\right]\right\}F(t).
 \end{align*}

Same as \eqref{LLo}, 
$$F(t) =\bigo(\iln_m(T_*+t)^{-\mu+s})\text{ for all } s>0.$$

Additionally, note, for $\ell\ge 1$, that $\iln_\ell(T_*+t)^{-1}=\bigo(1)$ as $t\to\infty$.

Therefore, it holds, for any $s>0$, that 
\beq\label{dF0}
 \left|\frac{\d F(t)}{\d t}\right|=\bigo\left( (T_*+t)^{-1} \iln_m(T_*+t)^{-\mu+s}\right).
  \eeq

\medskip\noindent
Consider $m=0$. Let $\gamma$ be an arbitrary number in $(0,1)$.
 Taking $s=1-\gamma>0$ in \eqref{dF0} yields
\beqs
\left| \frac{\d F(t)}{\d t}\right|=\bigo ((T_*+t)^{-\mu-\gamma}).
 \eeqs

By the continuity of $\d F(t)/\d t$ and $ (T_*+t)^{-\mu-\gamma}$ on $[0,\infty)$, we deduce
\beq\label{dF1}
\left| \frac{\d F(t)}{\d t}\right|\le C_2 (T_*+t)^{-\mu-\gamma} \text{ for all }t\ge 0,
 \eeq
for some positive constant $C_2$.
Using \eqref{dF1} in \eqref{JJ} and then applying Lemma \ref{plnlem} to estimate the resulting  integral gives
\beq\label{Jbo}
|J(t)|\le C_1C_2\int_0^t e^{-\lambda_1(t-\tau)/2}(T_*+\tau)^{-\mu-\gamma} \d\tau
 =\bigo((T_*+t)^{-\mu-\gamma}).
\eeq

Combining \eqref{Iint},  \eqref{Iz}, \eqref{Iini} and \eqref{Jbo} yields \eqref{log1}.

\medskip\noindent
Consider $m\ge 1$. Taking $s=\mu>0$ in \eqref{dF0} yields 
\beqs
\left|\frac{\d F(t)}{\d t}\right|=\bigo((T_*+t)^{-1}).
\eeqs 
 
We obtain, similar to \eqref{Jbo}, that
\beq\label{Jbo2}
J(t)=\bigo((T_*+t)^{-1}).
\eeq

Combining \eqref{Iint},  \eqref{Iz}, \eqref{Iini} and \eqref{Jbo2} yields \eqref{log2}. Then \eqref{log1} for $m\ge 1$ follows \eqref{log2}.

Finally, inequality \eqref{log3} clearly is a consequence of \eqref{log2}.
 \end{proof}

The main asymptotic approximation result for this subsection is the following.

\begin{theorem}\label{iterlog}
Given integers $m,k\in \Z_+$ with $k\ge m$, and a number $t_0>E_k(0)$.
Let  $\mu>0$, $p\in \mathscr P_{m}(k,-\mu,\C^n)$, and let function $g\in C([t_0,\infty),\C^n)$ satisfy 
\beq \label{gprop}
|g(t)|=\bigo(\iln_m(t)^{-\alpha})\text{ for some }\alpha>\mu.
\eeq 

Suppose $y\in C([t_0,\infty),\C^n)$ is a solution of 
 \beq\label{yeq4}
 y'=-Ay+p(\widehat \LL_k(t))+g(t)\text{ on } (t_0,\infty).
 \eeq

 Then  there exists  $\delta >0$  such that 
 \beq\label{ypL}
 |y(t)-(\mathcal Z_Ap)(\widehat \LL_k(t))|=\bigo(\iln_m(t)^{-\mu-\delta}).
 \eeq
 \end{theorem}
 \begin{proof}
 By the variation of constant formula, 
\beq\label{ytt}
y(t_0+t)=e^{-tA}y(t_0)+\int_{0}^{t} e^{-(t-\tau)A} p(\widehat\LL_k(t_0+\tau))\d\tau+\int_{0}^{t} e^{-(t-\tau)A}g(t_0+\tau)\d\tau.
\eeq

For the first term on the right-hand side of \eqref{ytt}, we have, by \eqref{eA2},  
\beqs 
e^{-tA}y(t_0)=\bigo(e^{-\lambda_1 t/2})=\bigo(\iln_m(t_0+t)^{-\mu-1}).
\eeqs 

For the second term on the right-hand side of \eqref{ytt},  we apply Lemma \ref{newplem} to $N=k$, $T_*=t_0$, using equation \eqref{log1} with $\gamma=1/2$ when $m=0$, and equation \eqref{log3} with $\gamma=\mu+1/2$ when $m\ge 1$. 
It results in 
\beqs
   \int_{0}^{t} e^{-(t-\tau)A} p(\widehat\LL_k(t_0+\tau))\d\tau=(\mathcal Z_Ap)(\widehat \LL_k(t_0+t))+\bigo(\iln_m(t_0+t)^{-\mu-1/2}).
\eeqs 

Consider  the third term on the right-hand side of \eqref{ytt}.
By \eqref{gprop} and the continuity of $g(t)$ and $\iln_m(t)>0$ on $[t_0,\infty)$, there is $C>0$ such that 
\beq\label{gbound}
|g(t)|\le C\iln_m(t)^{-\alpha}\quad\text{ for all $t\ge t_0$. }
\eeq

Then combining  inequalities \eqref{eA2}, \eqref{gbound} with  Lemma \ref{plnlem} yields
\beqs
\left| \int_{0}^{t} e^{-(t-\tau)A}g(t_0+\tau)\d\tau\right|
\le C_0C\int_{0}^{t} e^{-\lambda_1(t-\tau)/2}\iln_m(t_0+t)^{-\alpha}\d\tau =\bigo(\iln_m(t_0+t)^{-\alpha}).
\eeqs

Let $\delta=\min\{1/2,\alpha-\mu\}>0$. We obtain
\beq
y(t_0+t)=(\mathcal Z_Ap)(\widehat \LL_k(t_0+t))+\bigo(\iln_m(t_0+t)^{-\mu-\delta}),
\eeq
which proves \eqref{ypL}.
\end{proof}

Note that $q=\mathcal Z_Ap$ in \eqref{ypL} belongs to $\mathscr P_{m}(k,-\mu,\C^n)$, the same as $p$ on the right-hand side of equation \eqref{yeq4}. Estimate \eqref{ypL} shows that $q(\widehat \LL_k(t)$ is an asymptotic approximation of $y(t)$.
Such function $q(\widehat \LL_k(t)$ turns out to satisfy a specific ODE as shown in the next lemma.

\begin{lemma}\label{logode}
If $k\in \Z_+$ and $q\in \mathscr P(k,\C^n)$, then
 \beq\label{dq0}
\ddt q(\widehat \LL_k(t))=\mathcal M_{-1} q(\widehat \LL_k(t))+\mathcal Rq(\widehat \LL_k(t)) \text{ for }t>E_k(0).
 \eeq

In particular, when  $k\ge m\ge 1$, $\mu\in\R$, and $q\in\mathscr P_{m}(k,\mu,\C^n)$, one has
 \beq\label{dq1}
 \ddt q(\widehat \LL_k(t))=\mathcal M_{-1} q(\widehat \LL_k(t))+\bigo(t^{-\gamma})\quad
 \text{for all }\gamma\in(0,1).
 \eeq
 \end{lemma}
\begin{proof}
Because formulas \eqref{dq0} and \eqref{dq1} depend linearly on $q$,
it suffices to prove them for $q(z)= z^{\alpha}\xi$, for a constant vector $\xi\in\C^n$, where $z=(z_{-1},z_0,\ldots,z_k)$ and $\alpha=(\alpha_{-1},\alpha_0,\ldots,\alpha_k)$.

By the chain rule and \eqref{Lmderiv}, one has
\begin{align*}
\ddt q(\widehat \LL_k(t))
&=\alpha_{-1}z^{\alpha} \Big|_{z=\widehat\LL_k(t)} \xi 
+t^{-1}\left( 
\alpha_0 z^{\alpha}\Big|_{z=\widehat\LL_k(t)}
+ \sum_{j=1}^k\frac{1}{\iln_1(t)\ldots \iln_{j-1}(t)}\cdot \left.\frac{\alpha_j z^\alpha}{z_j}\right|_{z=\widehat\LL_k(t)}  
 \right )\xi\\
&=\alpha_{-1}q(\widehat \LL_k(t))+t^{-1}\left.  
\left(\mathcal M_0 q(z)
+  \sum_{j=1}^k\frac{1}{z_1\ldots z_j}\cdot\mathcal M_jq(z) \right )
\right|_{z=\widehat\LL_k(t)}.
\end{align*}
Therefore, we obtain \eqref{dq0}.

Consider $k\ge m\ge 1$ and $q\in\mathscr P_{m}(k,\mu,\C^n)$ now.   In this case, $\Re\alpha_{-1}=\Re\alpha_0=0$, which implies that $\mathcal M_jq$, for $0\le j\le k$, belongs to $\mathscr P_0(k,0,\C^n)$. Then, by applying \eqref{LLo} to $m=\mu=0$, one has
\beqs
R(t)\eqdef  \left.\left(\mathcal M_0 q(z) 
+\sum_{j=1}^k z_1^{-1}z_2^{-1}\ldots z_{j}^{-1}\mathcal M_jq(z)\right)\right|_{z=\widehat \LL_k(t)}=\bigo(t^s)
\text{ for any $s>0$.}
\eeqs
 Therefore,
\beq\label{Rqo}
\mathcal Rq(\widehat \LL_k(t))=t^{-1}R(t)=\bigo(t^{-\gamma})\quad\text{for all }\gamma\in(0,1).
\eeq
Then \eqref{dq1} follows \eqref{dq0} and \eqref{Rqo}.
\end{proof}

Note that if $k=-1$, then one can define $\mathcal R q=0$ and formula \eqref{dq0} still holds true.

\section{Global existence and asymptotic estimates}\label{odesec}

In this section, we establish the existence and uniqueness of a solution $y(t)$ of \eqref{sys-eq} for all $t\ge 0$ when the initial data and the forcing function are small. For a decaying solution, we obtain its specific decaying rates corresponding to that of the forcing function.

The following Theorems \ref{thmsmall}  and \ref{thmdecay} are counterparts of Theorems 3.1 and 3.2 in \cite{CaH3}. However, because of the changed assumption on the matrix $A$, the proofs presented below are different. They are based on a standard technique \cite{CL55} using the variation of constants formula instead of differential inequalities.
(See also \cite[Lemma 3]{Shi2000} for an autonomous system.)

\begin{theorem}\label{thmsmall}
There are positive numbers $\varep_0$ and $\varep_1$ such that if $y_0\in\C^n$  and 
$f\in C([0,\infty))$ satisfy 
\beq\label{fep} 
|y_0|<\varep_0 \text{ and }\|f\|_{\infty}:=\sup\{|f(t)|:t\in[0,\infty)\}<\varep_1,
\eeq 
then there exists a unique solution $y\in C^1([0,\infty))$ of \eqref{sys-eq} on $[0,\infty)$ with $y(0)=y_0$. 
Moreover, 
\beq\label{limyf}
\limsup_{t\to\infty}|y(t)|\le \frac{4C_0}{\lambda_1}\limsup_{t\to\infty}|f(t)|.
\eeq

Consequently, if 
\beq\label{limf} 
\lim_{t\to\infty}f(t)=0,
\eeq  
then 
 \begin{equation}\label{limy}
  \lim_{t\to\infty} y(t)=0.
 \end{equation}
 \end{theorem}
\begin{proof}
Recall that $C_0$ is the positive constant in \eqref{eA2}. We choose 
\beq\label{Mechoice}
M=\min \left\{ r_*,\frac{\lambda_1}{12C_0c_*} \right\},\quad 
\varep_0=\min \left \{\frac{M}{2},\frac{M}{6C_0} \right \},\quad 
\varep_1=\frac{\lambda_1 M}{12C_0}.
\eeq

Suppose that solution $y(t)$ exists on $[0,T)$, for some $T\in (0,\infty]$, and 
\beq\label{yasum}
|y(t)|< M\text{ for all } t\in [0,T).
\eeq

For $t\in[0,T)$, one has 
\beq\label{y0}
y(t)=e^{-tA}y_0+\int_0^t e^{-(t-\tau)A}(G(y(\tau))+f(\tau))\d\tau.
\eeq

Using \eqref{y0}, \eqref{eA2} and \eqref{Gyy}, we can estimate
\beq\label{y1}
|y(t)|\le C_0 e^{-\lambda_1 t/2} |y_0|+ C_0 \int_0^t e^{-\lambda_1(t-\tau)/2} (c_*|y(\tau)|^2+|f(\tau)|)\d\tau. 
\eeq

By the assumptions in \eqref{fep}, we have
\beq\label{y3}
|y(t)|
\le C_0e^{-\lambda_1t/2}\varep_0+ C_0 \int_0^t e^{-\lambda_1(t-\tau)/2}(c_*M^2+\varep_1)\d\tau
\le C_0\varep_0+\frac{2C_0}{\lambda_1}\left(c_*M^2+\varep_1\right) . 
\eeq

The last upper bounded in \eqref{y3} can be further estimated, thanks to \eqref{Mechoice},  by 
$$C_0\varep_0+\frac{2C_0c_*M}{\lambda_1}M+\frac{2C_0\varep_1}{\lambda_1} \le \frac{M}{6}+\frac{M}{6}+\frac{M}{6}=\frac{M}{2}.$$

By the standard contradiction argument with the choice of parameters in \eqref{Mechoice}, we obtain $T=\infty$, and \eqref{yasum} holds true.

\medskip
We prove \eqref{limyf} now. Let  $F=\limsup_{t\to\infty} |f(t)|$ and $L=\limsup_{t\to\infty} |y(t)|$.  By properties \eqref{fep} and \eqref{yasum} with $T$ already being established to be $\infty$, we have
\beqs
0\le F\le \| f \|_\infty\text{ and } 0\le L\le M.
\eeqs 

By taking the limit superior of \eqref{y1} as $t\to\infty$ and using \cite[Lemma 3.9]{HI1} to deal with integral, we obtain
\begin{align*}
L&\le \frac{2C_0}{\lambda_1}\left(c_* \limsup_{t\to\infty} |y(t)|^2+\limsup_{t\to\infty} |f(t)|\right)
= \frac{2C_0}{\lambda_1}(c_*L^2+F)\\
&\le \frac{2C_0c_*M}{\lambda_1}\cdot L+ \frac{2C_0}{\lambda_1}F.
\end{align*}

Because of \eqref{Mechoice}, one has  $2C_0c_* M/\lambda_1\le 1/6< 1/2$. Hence, it follows that
$L\le 4C_0F/\lambda_1$,
which proves \eqref{limyf}.

\medskip
Finally, under condition \eqref{limf}, one has $F=0$ which, by \eqref{limyf}, deduces \eqref{limy}. 
\end{proof}

Thanks to Theorem \ref{thmsmall}, the set of solutions $y(t)$ that satisfy \eqref{limy} is not empty.
When more information about \eqref{limf} is provided,  more specific form of \eqref{limy} will be obtained.  

\begin{theorem}\label{thmdecay}
Assume there is $T\ge 0$ such that $f\in C((T,\infty))$. 
 Let $y\in C^1((T,\infty))$ be a solution of \eqref{sys-eq} on $(T,\infty)$ that satisfies 
 \beq\label{limyy}
 \liminf_{t\to\infty}|y(t)|=0.
 \eeq  

\begin{enumerate}[label=\tnum]
\item \label{Th1}
 If there is a number $\alpha\in(0,\lambda_1)$ such that
\beq\label{fbigo}
f(t)= \bigo(e^{-\alpha t}),
\eeq
then
\beq\label{ylin1}
y(t)= \bigo(e^{-\alpha t}).
\eeq

\item\label{Th2} If there are numbers $m\in\Z_+$ and $\alpha>0$ such that 
\beq\label{fsmallo}
f(t)=o(\iln_m(t)^{-\alpha}),
\eeq
then
\beq\label{tlin2}
y(t)=\bigo(\iln_m(t)^{-\alpha}).
\eeq
\end{enumerate}
\end{theorem}
\begin{proof}
The proof is similar to that of Theorem \ref{thmsmall} with some refinements.

\medskip\noindent
\textit{Part} (i). Take $\varep=(\lambda_1-\alpha)/2>0$. Then
\beqs
\alpha<\lambda_1-\varep=\alpha+\varep=\frac12(\lambda_1+\alpha)<\lambda_1.
\eeqs

Let $c_\varep$ be a positive number  such that inequality \eqref{eAe} holds true.
Let 
\beq\label{Me2}
M=\min \left\{ r_*,\frac{\varep}{6c_\varep c_*} \right\},\quad
\varep_0=\min \left \{\frac{M}{2},\frac{M}{6c_\varep } \right \},\quad 
\varep_1=\frac{\varep M}{6c_\varep }.
\eeq

Fix $T_0>T$. Because of \eqref{fbigo} and the continuity of $f$ on $[T_0,\infty)$, there is $C>0$ such that 
\beq\label{f2}
| f(T_*+t)|\le C e^{-\alpha(T_*+ t)}\quad \text{ for any $T_*\ge T_0$ and $t\ge 0$.}
\eeq

By the virtue of \eqref{limyy}, there is a sufficiently large  $T_*\in[T_0,\infty)$ such that 
$$|y(T_*)|<\varep_0\text{ and }C e^{-\alpha T_*}\le \varep_1.$$
From this choice of $T_*$ and \eqref{f2}, it follows that
\beq\label{fe1}
| f(T_*+t)|\le \varep_1 e^{-\alpha t}\quad \text{ for all }t\ge 0.
\eeq

Set $z(t)=y(T_*+t)$ and $\widetilde f(t)= f(T_*+t)$. Then 
\beq\label{ztilf}
z'(t)+Az(t)=G(z)+\widetilde f(t),\text{ for } t\in[0,\infty).
\eeq

By the variation of constants formula, one has, for all $t\ge 0$,
\beq\label{zfT}
z(t)=e^{-tA}y(T_*) + \int_0^t e^{-(t-\tau)A}(G(z(\tau))+\widetilde f(\tau))\d\tau.
\eeq

Suppose 
\beq\label{yest}
|z(t)|< Me^{-\alpha t} \quad\text{ for all $t\in[0,T')$, where $T'\in(0,\infty]$.}
\eeq

For $t< T'$, by using \eqref{zfT}, inequality \eqref{eAe}, property \eqref{Gyy},  \eqref{fe1} and \eqref{yest}, one can obtain, similar to \eqref{y1}  and \eqref{y3},  
\begin{align*}
|z(t)|
&\le c_\varep e^{-(\lambda_1-\varep)t}|y(T_*)|
+   \int_0^t c_\varep e^{-(\lambda_1-\varep)(t-\tau)}(c_* M^2e^{-2\alpha\tau}+\varep_1e^{-\alpha \tau})\d\tau\\
&\le c_\varep e^{-\alpha t}\varep_0
+ c_\varep  \int_0^t e^{-(\alpha+\varep)(t-\tau)}(c_* M^2+\varep_1)e^{-\alpha \tau}\d\tau.
\end{align*}
This and  our choice in \eqref{Me2} yield 
\beqs
|z(t)|
\le \left(c_\varep \varep_0+ \frac{c_\varep(c_* M^2+\varep_1)}{\varep}\right)e^{-\alpha t}
\le \left( \frac M 6+\frac M 6+\frac M 6\right) e^{-\alpha t}
=\frac M 2 e^{-\alpha t}.
\eeqs

By a standard contradiction argument, it can be shown that \eqref{yest} holds with $T'=\infty$.
By shifting the time variable, it follows from \eqref{yest} that  $y(t)$ satisfies \eqref{ylin1}.

\medskip\noindent
\textit{Part} (ii). Set $\psi=L_m$.
We choose $T_*>\max\{T,E_{m+1}(0)\}$.  Thanks to properties \eqref{Linc} and \eqref{Lone},  the function $t\mapsto\psi(T_*+t)$ is increasing on $[0,\infty)$, and   $\psi(T_*+t)\ge 1$ for all $t\ge 0$.

There is a constant $C'_*>0$ such that
\beq\label{iine9}
e^{-\lambda_1 t/2} \le C'_*\psi(T_*+t)^{-\alpha} \quad\text{for all }t\ge 0.
\eeq

By inequality \eqref{iine2}, there is $C_*>0$ such that
\beq\label{iine3}
 \int_0^t e^{-\lambda_1(t-\tau)/2}\psi(T_*+\tau)^{-\alpha}\d\tau
 \le C_* \psi(T_*+t)^{-\alpha} \quad\text{for all }t\ge 0.
\eeq

Let $C_0>0$ be as in \eqref{eA2}.
Let 
\beq\label{Me3}
M=\min \left\{ r_*,\frac{\varep}{6C_0 C_* c_*} \right\},\quad
\varep_0=\min \left \{\frac{M}{2},\frac{M}{6C_0 C'_* } \right \},\quad 
\varep_1=\frac{\varep M}{6C_0C_*}.
\eeq

Thanks to conditions \eqref{limyy} and \eqref{fsmallo}, we can choose $T'_*>0$ sufficiently large so that $|y(T_*+T'_*)|<\varep_0$  and 
\beq\label{fTe}
|f(T_*+T'_*+t)|\le \varep_1 \psi(T_*+T'_*+t)^{-\alpha}\quad\text{ for all $t\ge 0$.}
\eeq

Let $z(t)=y(T_*+T'_*+t)$ and $\widetilde f(t)=f(T_*+T'_*+t)$. Then $z(t)$ satisfies \eqref{ztilf} again.

Because of \eqref{fTe} and the increase of $\psi(T_*+t)$ in $t\ge 0$, we have
\beqs
|\widetilde f(t)|\le \varep_1 \psi(T_*+t)^{-\alpha}\quad \text{ for all $t\ge 0$.}
\eeqs

Suppose
\beq\label{zest3}
|z(t)|\le M \psi(T_*+t)^{-\alpha} \quad\text{ for all }t \in[0,T'), \text{ where }T'\in(0,\infty].
\eeq

Performing the same calculations as in part (i) with the use of \eqref{eA2} in place of \eqref{eAe}, and  inequality \eqref{iine9},  one has, for $t\in[0,T')$,
\begin{align*}
|z(t)|
&\le |e^{-tA}||y(T_*+T'_*)|+  \int_0^t |e^{-(t-\tau)A}|(c_* |z(\tau)|^2+|\widetilde f(\tau)|)\d\tau\\
&\le C_0e^{-\lambda_1 t/2}\varep_0+ C_0 \int_0^t e^{-\lambda_1(t-\tau)/2}(c_* M^2\psi(T_*+\tau)^{-2\alpha}+\varep_1\psi(T_*+\tau)^{-\alpha})\d\tau\\
&\le C_0 C'_* \psi(T_*+t)^{-\alpha} \varep_0+ C_0 \int_0^t e^{-\lambda_1(t-\tau)/2}(c_* M^2+\varep_1)\psi(T_*+\tau)^{-\alpha}\d\tau.
\end{align*}

Applying \eqref{iine3} to estimate the last integral, we derive
\beqs
|z(t)| \le C_0(C'_*\varep_0 +C_*c_*M^2+C_* \varep_1) \psi(T_*+t)^{-\alpha}.
\eeqs

Together with \eqref{Me3}, this inequality gives 
\beqs
|z(t)|\le \frac M 2 \psi(T_*+t)^{-\alpha} .
\eeqs

Then \eqref{zest3} holds true with  $T'=\infty$. By shifting the time and using property \eqref{Lshift}, we obtain \eqref{tlin2} from \eqref{zest3}.
\end{proof}

\begin{remark}
Following the proof of Theorem \ref{thmdecay}, we, in fact,  can replace condition \eqref{limyy} with a weaker one, namely, 
 \beqs
 \liminf_{t\to\infty}|y(t)|<\varep_0.
 \eeqs 
 
Another technical point is that the small oh condition \eqref{fsmallo} is slightly stricter than the corresponding big oh condition (3.13) in \cite[Theorem 3.2]{CaH3}. Nonetheless, the result still serves our proofs well in Sections \ref{eforce}--\ref{lforce}. 
\end{remark}

\section{Case of exponential decay}\label{eforce}

We study the system of nonlinear differential equations \eqref{sys-eq}.

\begin{assumption}\label{assumpf}
There exists a number $T_f\ge 0$ such that $f$ is continuous on $[T_f,\infty)$.
\end{assumption}

More specific conditions on $f$ will be specified later for each result.

\begin{assumption}\label{assumpu}
There exists a number $T_0\ge 0$ such that $y\in C^1((T_0,\infty))$ is a solution of \eqref{sys-eq} on $(T_0,\infty)$,  and $y(t)\to 0$ as $t\to\infty$.
\end{assumption}

The main assumption on $f$ for this section is the following.

\begin{assumption} \label{fEas}
The function $f(t)$ admits the asymptotic expansion, in the sense of Definition \ref{EEdef} with  $X=\C^n$,
\beq\label{fmu}
f(t) \sim \sum_{k=1}^\infty f_k(t), \text{ where $f_k\in \mathcal F_E(-\mu_k,\C^n)$ for $k\in\N$, }
\eeq
with $(\mu_k)_{k=1}^\infty$ being a divergent, strictly increasing sequence of positive numbers.
Moreover, the set $\mathcal S\eqdef\{\mu_k:k\in\N\}$ preserves the addition and contains $\Re\sigma(A)$.
\end{assumption}

Note from the last condition of Assumption \ref{fEas} that $\mu_1\le \lambda_1$.

Under Assumption \ref{fEas}, denote 
\beqs
\bar f_N(t)=\sum_{k=1}^N f_k(t).
\eeqs 

Then, according to Definition \ref{EEdef}, for any $N\in\N$, one has
\beq \label{fN}
|f(t)-\bar f_N(t)|
=\Big|f(t)-\sum_{k=1}^N f_k(t)\Big|=\bigo(\psi(t)^{-\mu_{N}-\varep_N})\text{ for some } \varep_N>0.
\eeq

The following Scenarios \ref{scen1}--\ref{scen3} are typical cases that Assumption \ref{fEas} holds true.
They are of the same nature as Scenarios 4.6--4.8 in \cite{CaH3}.

\begin{scenario}\label{scen1}
Suppose the forcing function has the following expansion,  in the sense of Definition \ref{EEdef},
\beq\label{falpha}
f(t)\sim \sum_{k=1}^\infty \widetilde f_k(t), \text{ where $\widetilde f_k\in\mathcal F_E(-\alpha_k,\C^n)$ for $k\in\N$, }
\eeq
and $(\alpha_k)_{k=1}^\infty$ is a divergent, strictly increasing sequence of positive numbers.

We define $\mathcal S$ to be the additive semigroup generated by the real parts of $\Lambda_j$'s and $\alpha_j$'s, i.e.,
\beq\label{S1}
\mathcal S=\big\langle \Re \{ \Lambda_j,\alpha_\ell: 1\le j\le n,\ell\in \N  \}\big\rangle.
\eeq

We can re-arrange the set $\mathcal S$ as a sequence $(\mu_k)_{k=1}^\infty$ which is divergent and strictly increasing. The set $\mathcal S$ and sequence $(\mu_k)_{k=1}^\infty$ satisfy the properties in Assumption \ref{fEas}. 
Note that $\mathcal S$ contains the set $\{\alpha_k:k\in\N\}$, and, hence, $(\alpha_k)_{k=1}^\infty$ is a subsequence of $(\mu_k)_{k=1}^\infty$.
Then one can formally rewrite the sum in \eqref{falpha}, after re-indexing $\widetilde f_k$'s, as the sum in \eqref{fmu}, and verify that expansion \eqref{fmu}, indeed, holds true in the sense of Definition \ref{EEdef}.
\end{scenario}

\begin{scenario}\label{scen2}
Suppose $f(t)$ has the finite asymptotic expansion, in the sense of Definition \ref{EEdef}\ref{EE2},
\beq\label{ffin}
f(t)\sim \sum_{k=1}^N \widetilde f_k(t), \text{ where $\widetilde f_k\in\mathcal F_E(-\alpha_k,\C^n)$ for $1\le k\le N$.}
\eeq

Define the set $\mathcal S$ by formula \eqref{S1} with  $1\le \ell\le N$. This set $\mathcal S$ is still infinite, and can be arranged as sequence $(\mu_k)_{k=1}^\infty$ as in Scenario \ref{scen1}. Then, again, we can obtain \eqref{fmu} from  \eqref{ffin}.
\end{scenario}

\begin{scenario}\label{scen3}
Consider the case when the function $f$ is zero, or, more generally, decays faster than any exponential functions, that is, $e^{\alpha t} f(t)\to 0$ as $t\to\infty$ for any $\alpha>0$. 

Let $\mathcal S=\big\langle  \Re\sigma(A) \big\rangle$.
Again, arrange $\mathcal S$ as the sequence $(\mu_k)_{k=1}^\infty$. Then we have expansion \eqref{fmu} with $f_k=0$ for all $k\in\N$.
\end{scenario}

The main result of this section is the following.

\begin{theorem}\label{mainthm}
Under  Assumptions \ref{assumpf}, \ref{fEas} and \ref{assumpu},
 there exist functions
 \beq\label{yFE} 
 y_k\in \mathcal F_E(-\mu_k,\C^n) \text{ for $k\in\N$,}
 \eeq 
  such that the solution $y(t)$ admits the asymptotic expansion
   \beq \label{yykex}
  y(t)\sim \sum_{k=1}^\infty y_k(t) \text{  in the sense of Definition \ref{EEdef}.}
  \eeq 
  
 Moreover,  for each $k\in\N$, the functions $y_k(t)$ solves the following equation
\beq \label{ykeq}    
y_k'+ Ay_k  
=\sum_{m\ge 2}\ \sum_{\mu_{j_{1}}+\mu_{j_{2}}+\ldots \mu_{j_{m}}=\mu_k} \mathcal G_m(  y_{j_{1}}, y_{j_{2}},\ldots ,  y_{j_{m}}) +f_k, \text{ for }t\in\R.
\eeq  
\end{theorem}
 
The followings are clarifications for equation \eqref{ykeq}.
\begin{enumerate}[label={\rnum}]
 \item\label{Ta}  The indices $j_1,j_2,\ldots,j_m$ are taken to be in $\N$.
The double summation is zero if there are no indices that satisfy the stated conditions.

\item \label{T1}
When $k=1$, equation \eqref{ykeq} reads as
\begin{align} \label{y1eq}    
y_1'+Ay_1  = f_1.
\end{align}

\item\label{Tc}  We recall identity (4.20) of \cite{CaH3}, which states,  for any numbers $M\ge \mu_k/\mu_1$ and $Z\ge k-1$, that
 \beq\label{sumequiv}
\sum_{2\le m\le M}\sum_{\substack{{1\le j_{1},j_{2},\ldots,j_{m}\le Z}\\ \mu_{j_{1}}+\mu_{j_{2}}+\ldots \mu_{j_{m}}=\mu_k}}
=\sum_{m\ge 2}\sum_{\mu_{j_{1}}+\mu_{j_{2}}+\ldots \mu_{j_{m}}=\mu_k}.
\eeq

One can see from \eqref{sumequiv} that the double summation in \eqref{ykeq} is a finite sum.

\item\label{Tb} Regarding the double summation in \eqref{ykeq}, note that $m\ge 2$ and $\mu_{j_\ell}>0$ for $\ell=1,2,\ldots,m$. Hence,  $\mu_{j_\ell}<\mu_k$ which implies $j_\ell\le k-1$. Therefore, the terms $y_{j_\ell}$'s in the summand come from the previous steps.
 
\item\label{Te} Thanks to the previous properties \ref{T1} and \ref{Tb}, equation \eqref{ykeq} is actually a recursive expression.
\end{enumerate}
 
\begin{proof}[Proof of Theorem \ref{mainthm}]
Set $\psi(t)=e^t$ in this proof.

For $N\in\N$, denote by ($\mathcal T_N$) the statement: 
\textit{There exist functions $y_k\in \mathcal F_E(-\mu_k,\C^n)$, for  $k=1,2,\ldots,N$,  such that
\beq\label{hypoeq}
\text{ equation \eqref{ykeq} holds true on $\R$ for $k=1,2,\ldots,N$, }
\eeq
and
\beq\label{yremN}
\left|y(t)-\sum_{k=1}^N y_k(t) \right|=\bigo(\psi(t)^{-\mu_N-\delta_N})
\text{ for some $\delta_N>0$. }
\eeq
}

We will prove, by induction, that ($\mathcal T_N$) holds true for all $N\in\N$.
In the calculations below, $t$ will be sufficiently large.

\medskip
\noindent\textbf{First step.} Let  $N=1$.
By the triangle inequality, \eqref{fN} with $N=1$ and the fact $f_1\in \mathcal F_E(-\mu_1,\C^n)$, one has
\beqs
|f(t)| \le |f(t)- f_1(t)| + |f_1(t)| =\bigo(\psi(t)^{-\mu_1 - \varep_1}) + o(\psi(t)^{-\mu_1+\delta})
\eeqs
for all $\delta>0.$ This yields
\beq\label{ffirst} 
f(t)= o( \psi(t)^{-\mu_1+\delta}) \quad \forall \delta>0.
\eeq 

Applying Theorem \ref{thmdecay}\ref{Th1} yields 
\beq\label{u-first}
y(t)=\bigo(\psi(t)^{-\mu_1+\delta})\quad\forall\delta\in(0,\mu_1), \text{ hence, }\forall \delta>0.
\eeq

Select $\delta<\mu_1$ in \eqref{u-first} so that $2(\mu_1-\delta)>\mu_1+\varep_1$. 
Thanks to \eqref{u-first}, $|y(t)|<r_*$ for $t$ sufficiently large, hence, we can apply inequality \eqref{Gyy} to $x=y(t)$ and have
\beqs
G(y(t))=\bigo(|y(t)|^2)=\bigo(\psi(t)^{-2(\mu_1-\delta)})=\bigo(\psi(t)^{-\mu_1-\varep_1}).
\eeqs

We rewrite equation \eqref{sys-eq} as
 \beq\label{yf1}
 y'(t)+Ay(t)=f_1(t)+[G(y(t))+(f(t)-f_1(t))]=f_1(t)+\bigo(\psi(t)^{-\mu_1-\varep_1}).
 \eeq

By the virtue of Theorem \ref{approx1} applied to equation \eqref{yf1}, there exist a function $y_1\in\mathcal F_E(-\mu_1,\C^n)$ and a number $\delta_1 >0$ such that $y_1(t)$ satisfies
\eqref{y1eq}, which is \eqref{ykeq} for $k=1$, on $\R$, and
\beqs
  |y(t)-y_1(t)|= \bigo(\psi(t)^{-\mu_1-\delta_1 }).
\eeqs

Therefore, statement ($\mathcal T_1$) holds true. Notice that we did not check condition \eqref{ylimcond} because  $\Re\Lambda_j\ge \mu_1$ for all $j$.

\medskip
\noindent\textbf{Induction step.} Let $N\ge 1$.
Suppose there are function $y_k\in \mathcal F_E(-\mu_k,\C^n)$, for $1\le k\le N$, such that \eqref{hypoeq} and \eqref{yremN} are true.

Let
\beq \label{ys}
u_N(t)=\sum_{k=1}^N y_k(t) \text{ and } 
v_N(t)=y(t)-u_N(t).
\eeq 

The assumption \eqref{yremN} reads as
\beq\label{vNrate}
v_N(t)=\bigo(\psi(t)^{-\mu_N-\delta_N}).
\eeq
For $1\le k\le N+1$, define
\beq\label{Jk}
\mathcal J_k(t)=\sum_{m\ge 2}\ \sum_{\mu_{j_1}+\mu_{j_2}+\ldots \mu_{j_m}=\mu_k} \mathcal G_m(y_{j_1}(t),y_{j_2}(t),\ldots , y_{j_m}(t)).
\eeq

We claim that $v_N(t)$ satisfies
\beq\label{vNE}
v_N'(t)+Av_N(t) =f_{N+1}(t)
 +  \mathcal J_{N+1}(t)
 +\bigo(\psi(t)^{-\mu_{N+1}-\delta_*} ),
\eeq
for some number $\delta_*>0$.

We accept \eqref{vNE} momentarily. By  the fact that $y_{j_\ell}\in \mathcal F_E(-\mu_{j_\ell},\C^n)$ for $\ell=1,2,\ldots,m$ and $\mathcal G_m$ being multi-linear, it is clear that 
$$\mathcal G_m( y_{j_{1}}(t), y_{j_{2}}(t),\ldots , y_{j_{m}}(t))\in \mathcal F_E\Big(-\sum_{\ell=1}^m\mu_{j_\ell},\C^n\Big).$$

As a consequence,  $\mathcal J_{N+1}(t)\in\mathcal F_E(-\mu_{N+1},\C^n)$.
 We apply Theorem \ref{approx1} to  equation \eqref{vNE} and solution $v_N(t)$.

We check condition \eqref{ylimcond} for $\mu=\mu_{N+1}$. Suppose $\lambda=\Re\Lambda_j<\mu_{N+1}$.

Since $\Lambda_j$ is an eigenvalue of $A$, we have, thanks to the last condition of Assumption \ref{fEas}, $\lambda\in\mathcal S=\{\mu_k:k\in\N\}$.   This fact, together with $\lambda<\mu_{N+1}$ and $(\mu_k)_{k=1}^\infty$ being strictly increasing implies that $\lambda\le\mu_N$.
Combining this upper bound of $\lambda$ with the estimate of $v_N(t)$ in \eqref{vNrate}, we see that condition \eqref{ylimcond} is met for $y=v_N$ and any $m\in\N$. 

Then, by the virtue of Theorem \ref{approx1}, there exist a function $y_{N+1}\in \mathcal F_E(-\mu_{N+1},\C^n)$
 and a number $\delta_{N+1}>0$ such that
\beqs
|v_N(t)-y_{N+1}(t)|=\bigo(\psi(t)^{-\mu_{N+1}-\delta_{N+1}})
\eeqs
and 
\beqs
y_{N+1}'(t)+Ay_{N+1}(t)=f_{N+1}(t)+ \mathcal J_{N+1}(t)\text{ for }t\in\R.
\eeqs

Thus, the statement ($\mathcal T_{N+1}$) holds true.

\medskip
\noindent\textbf{Conclusion of the proof of ($\mathcal T_N$).} By the Induction Principle, the statement  ($\mathcal T_N$) is true for all $N\in\N$. 

\medskip
Note, for each $N\in\N$, that the function $y_{N+1}$ is constructed without changing the previous $y_k$ for $1\le k \le N$. Therefore, the functions $y_k$'s exist for all $k\in\N$, and,  for any $N\in\N$,   \eqref{hypoeq} and \eqref{yremN} in  ($\mathcal T_N$) hold true.
Consequently, we obtain the asymptotic expansion \eqref{yykex} with the functions $y_k$'s satisfying \eqref{ykeq}.

\medskip
It remains to prove  equation \eqref{vNE}. Its proof is divided into three parts (a)--(c).

\medskip
\noindent \textit{Part} (a). Note from \eqref{fEorder} that 
\beq\label{ykdelta}
|y_k(t)|=\bigo(\psi(t)^{-\mu_k+\delta})\quad\forall \delta>0, \quad k=1,2,\ldots,N.
\eeq
Then
\beq\label{yNbrate}
|u_N(t)|\le \sum_{k=1}^N |y_k(t)|=\bigo(\psi(t)^{-\mu_1+\delta})\quad\forall \delta>0.
\eeq

Taking derivative of $v_N(t)$ and using \eqref{sys-eq} give
\beqs
v_N'
=y'-\sum_{k=1}^N y_k'
=-Ay +G(y)+f(t) - \sum_{k=1}^N y_k'.
\eeqs 

By writing
\beqs
Ay=\sum_{k=1}^N Ay_k+Av_N\text{ and }
f(t)= \sum_{k=1}^N  f_k(t)+f_{N+1}(t)+\bigo(\psi(t)^{-\mu_{N+1}-\varep_{N+1}}),
\eeqs
 we have
\beq\label{vN4}
v_N'
= -Av_N+G(y)+f_{N+1}(t)- \sum_{k=1}^N \Big(Ay_k+y_k' -f_k(t)\Big)+\bigo(\psi(t)^{-\mu_{N+1}-\varep_{N+1}}).
\eeq 

By the induction hypothesis,
$y_k'=-Ay_k + \mathcal J_k(t)+f_k(t)$ for $k=1,\ldots,N$.
Hence, we obtain from \eqref{vN4} that
\beq\label{vNeq1}
v_N'=-Av_N+f_{N+1} +G(y) - \sum_{k=1}^N \mathcal J_k(t) + \bigo(\psi(t)^{-\mu_{N+1}-\varep_{N+1}}) .
\eeq

\medskip
\noindent \textit{Part} (b). We calculate $G(y(t))$. 
Letting $\delta=\mu_1/2$ in \eqref{u-first} yields
\beq \label{yr}
y(t)= \bigo( \psi(t)^{-\mu_1/2}).
\eeq

Let integer 
\beq\label{MN}
M_{N+1}\ge 2\mu_{N+1}/\mu_1.
\eeq

Note that $ M_{N+1}\ge 2$. By \eqref{Grem1}, there exists  $\theta_N>0$ such that
 \beq\label{Grem3}
 \Big|G(y)-\sum_{m=2}^{M_{N+1}} G_m(y)\Big|=\bigo(|y|^{M_{N+1}+\theta_N})\text{ as } y\to 0.
 \eeq
 
We calculate and estimate, using \eqref{Grem3} and \eqref{yr},
\begin{align*}
G(y(t))
&= \sum_{m=2}^{M_{N+1}} G_m(y(t)) + \bigo(|y(t)|^{M_{N+1}+\theta_N})\\
&=  \sum_{m=2}^{M_{N+1}} G_m(y(t)) +  \bigo(\psi(t)^{-(M_{N+1}+\theta_N)\mu_1/2}).
\end{align*}
Thus, thanks to \eqref{MN},
\beq\label{preG}
G(y(t))
= \sum_{m=2}^{M_{N+1}} G_m(y(t)) +  \bigo\left(\psi(t)^{-\mu_{N+1}-\theta_N\mu_1/2}\right).
\eeq

For each $G_m(y(t))$ in \eqref{preG}, we rewrite it, using \eqref{GG}, as 
\beq\label{GGm}
G_m(y(t))= G_m(u_N+v_N) 
 =\mathcal G_m((u_N+v_N)^{(m)}).
\eeq 

By the multi-linearity of $\mathcal G_m$ and inequality \eqref{multineq}, we have
\begin{align*}
G_m(y(t))
 &=\mathcal G_m(u_N^{(m)})+\mathcal G_m(v_N^{(m)})+\sum_{k=1}^{m-1}\bigo(|u_N(t)|^k|v_N(t)|^{m-k})\\
 &=\mathcal G_m(u_N^{(m)}) +\bigo(|v_N(t)|^2) +\bigo(|u_N(t)| |v_N(t)|).
\end{align*}

The last two terms are estimated, by using \eqref{vNrate} and \eqref{yNbrate} with $\delta=\delta_N/2$, by
\beq\label{OO}
\bigo(\psi(t)^{-2(\mu_N+\delta_N)})+\bigo(\psi(t)^{-\mu_1+\delta_N/2}\psi(t)^{-\mu_N-\delta_N }).
\eeq

Because $2\mu_N$ and $\mu_N+\mu_1$ are two numbers in $\mathcal S$ which are greater than $\mu_N$, they, in fact, are greater or equal to $\mu_{N+1}$.
Thus, the quantities in \eqref{OO} can be estimated further by
$$\bigo(\psi(t)^{-\mu_{N+1}-2\delta_N})+\bigo(\psi(t)^{-\mu_{N+1}-\delta_N/2}).$$

Therefore, we obtain
\beq\label{Gmb}
 G_m(y(t))=\mathcal G_m((u_N(t))^{(m)})+\bigo(\psi(t)^{-\mu_{N+1}-\delta_N/2}).
\eeq

Summing up \eqref{Gmb} in $m$ and combining with \eqref{preG}, we obtain
\beq\label{psG}
G(y(t))
= \sum_{m=2}^{M_{N+1}} \mathcal G_m((u_N(t))^{(m)})+\bigo(\psi(t)^{-\mu_{N+1}-\delta_N/2 })+  \bigo\left(\psi(t)^{-\mu_{N+1}-\theta_N\mu_1/2}\right).
\eeq

We continue to manipulate
\beq\label{sGm}
\sum_{m=2}^{M_{N+1}} \mathcal G_m((u_N(t))^{(m)})
=  \sum_{m=2}^{M_{N+1}}\sum_{1\le j_{1},j_{2},\ldots,j_{m}\le N} \mathcal G_m( y_{j_{1}}(t), y_{j_{2}}(t),\ldots , y_{j_{m}}(t)).
\eeq 

Note from \eqref{multineq}, the fact that each $y_{j_\ell}\in \mathcal F_E(-\mu_{j_\ell},\C^n)$ and the estimates in \eqref{ykdelta} that
\beq\label{Gq}
|\mathcal G_m( y_{j_{1}}(t), y_{j_{2}}(t),\ldots , y_{j_{m}}(t))|
\le \|\mathcal G_m\| \cdot \prod_{\ell=1}^m |y_{j_\ell}(t)|=\bigo\left(\psi(t)^{-(\mu_{j_1}+\mu_{j_2}+\ldots+\mu_{j_m})+\delta}\right)
\quad\forall \delta>0.
\eeq

Thanks to Assumption \ref{fEas}, the set $\mathcal S$ preserves the addition. Hence, the sum 
$$\mu_{j_{1}}+\mu_{j_{2}}+\ldots \mu_{j_{m}}\text{ belongs to $\mathcal S$,}$$
 and, thus, it must be $\mu_k$ for some $k\ge 1$.
Therefore, we can split the sum in \eqref{sGm} into three parts:
\beqs 
\mu_{j_{1}}+\mu_{j_{2}}+\ldots \mu_{j_{m}}=\mu_k\text{ for }
k\le N,\ 
k=N+1\text{ and }k\ge N+2.
\eeqs 

Corresponding to the last part, i.e., $\mu_k\ge \mu_{N+2}$, taking into account \eqref{Gq}, the summand in \eqref{sGm} is 
\beqs 
\mathcal G_m( y_{j_{1}}(t), y_{j_{2}}(t),\ldots , y_{j_{m}}(t))
=\bigo(\psi(t)^{-\mu_{N+2}+\delta })
\quad\forall \delta>0.
\eeqs 
Thus, we rewrite \eqref{sGm} as
\beq\label{sGm2}
\sum_{m=2}^{M_{N+1}} \mathcal G_m((u_N(t))^{(m)})
= \sum_{k=1}^{N+1} Q_k(t)  +\bigo(\psi(t)^{-\mu_{N+2}+\delta })\quad\forall \delta>0,
\eeq 
where, for $1\le k\le N+1$, 
\beqs
Q_k(t)= \sum_{m=2}^{M_{N+1}}\sum_{\substack{{1\le j_{1},j_{2},\ldots,j_{m}\le N}\\ \mu_{j_{1}}+\mu_{j_{2}}+\ldots \mu_{j_{m}}=\mu_k}} \mathcal G_m(y_{j_{1}}(t),y_{j_{2}}(t),\ldots ,q_{j_{m}}(t))
\quad\text{ for $t\in\R$.}
\eeqs

Therefore, combining  \eqref{psG} with  \eqref{sGm2} for $\delta=(\mu_{N+2}-\mu_{N+1})/2$  yields
\beq\label{Gy}
G(y(t))= \sum_{k=1}^{N+1} Q_k(t)+\bigo(\psi(t)^{-\mu_{N+1}-\delta_{N+1}'}),
\eeq
where $\delta_{N+1}'=\min\{\delta_N/2,\theta_N\mu_1/2, (\mu_{N+2}-\mu_{N+1})/2\}>0$.

\medskip
\noindent \textit{Part} (c). 
For $1\le k\le N+1$, we note that $N\ge k-1$, and, thanks to \eqref{MN}, $M_{N+1}>\mu_{N+1}/\mu_1>\mu_k/\mu_1$.
By relation \eqref{sumequiv},
\beqs 
\sum_{m=2}^{M_{N+1}}\sum_{\substack{{1\le j_{1},j_{2},\ldots,j_{m}\le N}\\ \mu_{j_{1}}+\mu_{j_{2}}+\ldots \mu_{j_{m}}=\mu_k}}
=\sum_{m\ge 2}\sum_{\mu_{j_{1}}+\mu_{j_{2}}+\ldots \mu_{j_{m}}=\mu_k}.
\eeqs 
Thus,  we have
\beq \label{QJ}
Q_k= \sum_{m\ge2}\sum_{\mu_{j_{1}}+\mu_{j_{2}}+\ldots \mu_{j_{m}}=\mu_k} \mathcal G_m(y_{j_{1}},y_{j_{2}},\ldots ,y_{j_{m}})=\mathcal J_k
\quad\text{ for }1\le k\le N+1.
\eeq

By \eqref{vNeq1}, \eqref{Gy} and \eqref{QJ}, we have 
\begin{align*}
v_N'+Av_N
&= f_{N+1}(t)+  \sum_{k=1}^{N+1} \mathcal J_k-\sum_{k=1}^N \mathcal J_k +\bigo(\psi(t)^{-\mu_{N+1}-\varep_{N+1}})+\bigo(\psi(t)^{-\mu_{N+1}-\delta_{N+1}'}).
\end{align*}

Therefore, we obtain the desired equation \eqref{vNE} with $\delta_*=\min\{\varep_{N+1},\delta_{N+1}'\}>0$, and completes the proof of Theorem \ref{mainthm}.
\end{proof}

\begin{remark}
 In the case $G$ has only a finite sum approximation and/or $f(t)$ has a finite sum approximation, the solution $y(t)$ admits a corresponding finite sum approximation, see the treatments in \cite[Theorem 2.6]{HM2}, \cite[Section 4.1]{CaH1} and \cite[Theorem 5.1]{CaHK1}.
\end{remark}

\section{Case of power-decay}\label{pforce}

In this section, we deal with the forcing functions that are power-decaying as time tends to infinity.

\begin{assumption} \label{fLas}
The function $f(t)$ admits the asymptotic expansion, in the sense of Definition \ref{Lexpand} with  $X=\C^n$ and $m_*=0$,
\beq\label{fmu2}
f(t) \sim \sum_{k=1}^\infty f_k(t),\text{ where   $f_k\in \mathscr F_{0}(n_k,-\mu_k,\C^n)$ for  $k\in\N$,}
\eeq
with  $(\mu_k)_{k=1}^\infty$ being a divergent, strictly increasing sequence of positive numbers,
and   $(n_k)_{k=1}^\infty$ being an increasing sequence in $\Z_+$.
Moreover, the set $\mathcal S\eqdef \{\mu_k:k\in\N\}$  preserves the addition and  the unit increment.
\end{assumption}

\begin{scenario}\label{scen4}
Suppose 
\beq\label{fal2}
f(t) \sim  \sum_{k=1}^\infty \widetilde f_k(t),\text{ where   $\widetilde f_k\in \mathscr F_{0}(\widetilde n_k,-\alpha_k,\C^n)$ for  $k\in\N$,}
\eeq
with  the sequence   $(\alpha_k)_{k=1}^\infty$ being the same as in Scenario \ref{scen1},
and $(\widetilde n_k)_{k=1}^\infty\subset \Z_+$.

 Let $\mathcal S$ be defined by
\beq\label{S2}
\mathcal S=\left \{ k+\sum_{j=1}^m \alpha_{\ell_j}: k\in\Z_+,\ m\in\N,\ \ell_j\in \N \right \}.
\eeq

Then the set $\mathcal S$ is infinite, preserves the addition and the unit increment.
We can arrange $\mathcal S$ to be a sequence $(\mu_k)_{k=1}^\infty$ as in Assumption \ref{fLas}, and take 
\beq \label{nk} n_k=\max\{\widetilde n_j:1\le j\le k\}.
\eeq

Same as in Scenario \ref{scen1} and taking into account the embedding in \ref{Cc} after Definition \ref{Fclass},
we obtain the asymptotic expansion \eqref{fmu2} from \eqref{fal2}.

Note in \eqref{S2} that, because of the fact $m\ge 1$, there is the presence of at least one $\alpha_{\ell_j}$, hence,  $\mu_1\ge \alpha_1$.  Then taking $k=0$ and  $m=j=\ell_j=1$ gives
\beq\label{mua1} \mu_1=\alpha_1.\eeq
\end{scenario}

A counterpart of Scenario \ref{scen2}  can also be similarly demonstrated.

\begin{theorem}\label{mainthm2}
Under Assumptions \ref{assumpf}, \ref{fLas}  and  \ref{assumpu},
 there exist functions 
 \beq \label{yF0}
 y_k\in \mathscr F_{0}(n_k,-\mu_k,\C^n)\text{ for $k\in\N$,}
 \eeq 
  such that the solution  $y(t)$ admits the asymptotic expansion  
   \beq \label{yykex2}
  y(t)\sim \sum_{k=1}^\infty y_k(t) \text{ in the sense of Definition \ref{Lexpand} with $m_*=0$.}
  \eeq
  
More specifically, assume, for all $k\in\N$,  
\beq \label{fkLp}
f_k(t)=p_k(\widehat \LL_{n_k}(t))\text{  for some $p_k\in\mathscr P_{0}(n_k,-\mu_k,\C^n)$.}
\eeq  

Then the functions $y_k$'s in \eqref{yykex2} can be constructed recursively as follows. For each $k\in\N$,
\beq \label{ykLq}
y_k(t)=q_k(\widehat \LL_{n_k}(t)),
\eeq 
where 
\beq \label{ykdirect}    
q_k  
=\mathcal Z_A\left( \sum_{m\ge 2}\ \sum_{\mu_{j_{1}}+\mu_{j_{2}}+\ldots \mu_{j_{m}}=\mu_k} \mathcal G_m( q_{j_{1}}, q_{j_{2}},\ldots ,  q_{j_{m}}) +p_k-\chi_k\right ),
\eeq  
with, recalling $\mathcal R$ defined by \eqref{chiz}, 
\beq\label{chikdef}
\chi_k=\begin{cases}
\mathcal R q_\lambda &\text{if there exists $\lambda\le k-1$ such that $\mu_\lambda+1=\mu_k$,}\\
0&\text{otherwise.}
    \end{cases}
\eeq
\end{theorem}

The following explanations and remarks are in order.
\begin{enumerate}[label=\rnum]
\item\label{8a} Certainly, $\chi_1=0$ in \eqref{chikdef} and $q_1=\mathcal Z_A p_1$.

\item \label{8c} Same as \ref{Te} after Theorem \ref{mainthm}, equation \eqref{ykdirect} is a recursive formula.

\item\label{8d} In \eqref{chikdef}, the index $\lambda$, if exists, is unique. Moreover, if $q_\lambda\in \mathscr P_{0}(n_\lambda,-\mu_\lambda,\C^n)$, then, by property \ref{R2} after Definition \ref{defMRZ}, then the  $\chi_k$ belongs to 
$$\mathscr P_{0}(n_\lambda,-\mu_\lambda-1,\C^n)=\mathscr P_{0}(n_\lambda,-\mu_k,\C^n)\subset \mathscr P_{0}(n_k,-\mu_k,\C^n).$$

\item\label{8b} By induction, one can verify, for all $k\in\N$,  that $q_k$ and $\chi_k$ defined by \eqref{ykdirect} and \eqref{chikdef}, respectively,  belong to $\mathscr P_{0}(n_k,-\mu_k,\C^n)$, and, hence, $y_k$ defined by \eqref{ykLq} belongs to $\mathscr F_{0}(n_k,-\mu_k,\C^n)$.
\end{enumerate}

\begin{proof}[Proof of Theorem \ref{mainthm2}]
We follow the proof of Theorem \ref{mainthm}.
Let $f_k$ be as in \eqref{fkLp} and $y_k$ be as in \eqref{ykLq}, for $k\in\N$.
Let $m_*=0$ and $\psi(t)=L_{m_*}(t)=t$ for $t>0$.

We will prove by induction that, for any $N\in\N$,
\beq\label{PTN}
\Big|y(t)-\sum_{k=1}^N y_k(t) \Big|=\bigo(\psi(t)^{-\mu_N-\delta_N})
\text{ for some $\delta_N>0$. }
\eeq

\medskip
\noindent\textbf{First step.} Let $N=1$. We similarly obtain estimate \eqref{ffirst} for $f(t)$ and then, by Theorem \ref{thmdecay}\ref{Th2}, 
estimate \eqref{u-first} for $y(t)$.
Continuing with the proof of Theorem \ref{mainthm} after that, we obtain equation \eqref{yf1}. 
 
Note that $q_1=\mathcal Z_Ap_1$ and $y_1(t)= q_1(\widehat{\LL}_{n_1}(t))$. Then $y_1\in \mathscr F_{0}(-\mu_1,\C^n)$, and, 
according to Theorem \ref{iterlog} applied to $p=p_1$ and $k=n_1$,
\beqs
|y(t)-y_1(t)|=\bigo(\psi(t)^{-\mu_1-\delta_1}).
\eeqs

Thus, \eqref{PTN} is true for $N=1$.

\medskip
\noindent\textbf{Induction step.} Let $N\ge 1$. Suppose that \eqref{PTN} holds true.

Let $u_N$ and $v_N$ be the same as in \eqref{ys}.
Then \eqref{PTN} reads as
$v_N=\bigo(\psi(t)^{-\mu_k-\delta_k})$.

We, again, obtain equation \eqref{vN4}.

Let $\mathcal J_k(t)$ be defined by \eqref{Jk} for $k\in\N$.
Treating $G(y)$ in the same way as in parts (b) and (c) in  the proof of Theorem \ref{mainthm}, we obtain \eqref{Gy} and \eqref{QJ}.
They imply
\beq\label{Gy2}
G(y(t))
=\sum_{k=1}^{N+1}\mathcal J_k(t) +\bigo(\psi(t)^{-\mu_{N+1}-\delta_{N+1}'}).
\eeq  

By formula \eqref{dq0},  it holds, for $k\in\N$, that
\beq \label{ykeq2}    
y_k'=(\mathcal M_{-1}q_k+\mathcal R q_k)\circ \widehat \LL_{n_k}\text{ on } (E_{n_k}(0),\infty).
\eeq  

Summing up \eqref{ykeq2} in $k$ gives
\beq\label{shorty}
\sum_{k=1}^N y_k'=\sum_{k=1}^N \mathcal M_{-1}q_k\circ \widehat\LL_{n_k} + \sum_{\lambda=1}^N \mathcal R q_\lambda\circ \widehat\LL_{n_\lambda} \text{ on } (E_{n_N}(0),\infty).
\eeq
Note that we already made a change of the index notation from $k$ to $\lambda$ for the last sum.

Regarding  the last sum in \eqref{shorty},
we observe that $\mathcal Rq_\lambda\in \mathscr F_{0}(n_\lambda,-\mu_\lambda-1,\C^n)$. Thanks to Assumption \ref{fLas},  $\mu_\lambda+1\in\mathcal S$. Hence, there exists a unique number $k\in\N$ such that $\mu_k=\mu_\lambda+1$. Because $\mu_k>\mu_\lambda$, we have $\lambda\le k-1$. Thus, $\mathcal Rq_\lambda=\chi_k$.
Consider three possibilities $k\le N$, $k=N+1$ and $k\ge N+2$, we  rewrite, similar to \eqref{sGm2},
\beq\label{sumR}
\sum_{\lambda=1}^N \mathcal R q_\lambda\circ \widehat\LL_{n_\lambda}
=\sum_{k=1}^N \chi_k\circ \widehat\LL_{n_k} +\chi_{N+1}\circ \widehat\LL_{n_{N+1}}+\bigo(\psi(t)^{-\mu_{N+2}+\delta})
\quad \forall \delta>0.
\eeq

Combining \eqref{vN4}, \eqref{Gy2} and \eqref{sumR} yields 
\beq\label{vN3}
v_N'+Av_N
= f_{N+1}(t) - \sum_{k=1}^N X_k(t) - \chi_{N+1}\circ \widehat\LL_{n_{N+1}}(t)+ \mathcal J_{N+1}(t) +\bigo(\psi(t)^{-\mu_{N+1}-\delta_*}),
\eeq 
for some $\delta_*>0$, where
\beqs
X_k(t)= (Aq_k+\mathcal M_{-1}q_k + \chi_k-p_k)\circ \widehat\LL_{n_k}(t)-\mathcal J_k(t).
\eeqs

For $k\in\N$ and $z\in (0,\infty)^{n_k+2}$, let 
\beqs
\mathcal Q_k(z)= \sum_{m\ge2}\sum_{\mu_{j_{1}}+\mu_{j_{2}}+\ldots \mu_{j_{m}}=\mu_k}  \mathcal G_m(q_{j_{m,1}}(z),q_{j_{m,2}}(z),\ldots , q_{j_{m,m}}(z)).
\eeqs

Obviously,  $\mathcal Q_k(\widehat\LL_{n_k}(t))=\mathcal J_k(t)$ for all $k\in\N$.
By identity \eqref{ZAM}, we can write
\beqs
\chi_k=(A+\mathcal M_{-1})\mathcal Z_A\chi_k,\
p_k=(A+\mathcal M_{-1})\mathcal Z_A p_k,\
\mathcal J_k=((A+\mathcal M_{-1})\mathcal Z_A\mathcal Q_k)\circ \widehat\LL_{n_k}.
\eeqs

Therefore,
\beqs
X_k(t)=\Big [(A+\mathcal M_{-1})( q_k + \mathcal Z_A(\chi_k-p_k-\mathcal Q_k))\Big] \circ \widehat\LL_{n_k}(t).
\eeqs 

For $1\le k\le N$, one has from \eqref{ykdirect} that
$q_k = \mathcal Z_A(\mathcal Q_k+p_k-\chi_k )$,
hence, $X_k=0$.
It follows from this fact and equation \eqref{vN3} that
\beq\label{vN5}
v_N'+Av_N 
=(p_{N+1}- \chi_{N+1}+ \mathcal Q_{N+1})\circ \widehat\LL_{n_{N+1}} (t) 
 +\bigo(\psi(t)^{-\mu_{N+1}-\delta_*} ).
\eeq

Applying Theorem \ref{iterlog} to equation \eqref{vN5} yields
\beqs
|v_N(t)-y_{N+1}(t)|=\bigo(\psi(t)^{-\mu_{N+1}-\delta_{N+1}}),
\eeqs
for some number $\delta_{N+1}>0$, where
\beqs
y_{N+1}=(\mathcal Z_A ( \mathcal Q_{N+1} + p_{N+1} -\chi_{N+1}))\circ \widehat\LL_{n_{N+1}}=q_{N+1}\circ \widehat\LL_{n_{N+1}}.
\eeqs

Thus, \eqref{PTN} is true for $N:=N+1$.

\medskip\noindent\textbf{Conclusion of the proof of \eqref{PTN}.}
By the Induction Principle, statement  \eqref{PTN}  is true for all $N\in\N$.

\medskip
Now, the asymptotic expansion \eqref{yykex2} clearly follows \eqref{PTN}. The proof of Theorem \ref{mainthm2} is complete.
\end{proof}

\section{Case of logarithmic or iterated logarithmic decay}\label{lforce}

We obtain the asymptotic expansions when the forcing function has logarithmic or iterated logarithmic decay.

\begin{assumption} \label{fLas2}
There exist a number $m_*\in\N$, a divergent, strictly increasing sequence $(\mu_k)_{k=1}^\infty\subset (0,\infty)$, and 
an increasing sequence $(n_k)_{k=1}^\infty\subset \N\cap[m_*,\infty)$ such that
the function $f(t)$ admits the asymptotic expansion, in the sense of Definition \ref{Lexpand} with  $X=\C^n$, 
\beq\label{fmu3}
f(t) \sim \sum_{k=1}^\infty f_k(t),\text{ where  $f_k\in \mathscr F_{m_*}(n_k,-\mu_k,\C^n)$ for $k\in\N$.} 
\eeq
Moreover, the set $\mathcal S\eqdef \{\mu_k:k\in\N\}$  preserves the addition. 
\end{assumption}

\begin{scenario}
Assume, similar to \eqref{fal2}  in Scenario \ref{scen4}, the following  asymptotic expansion, in the sense of Definition \ref{Lexpand} with $X=\C^n$ and $m_*\in\N$, 
\beq\label{fal3}
f(t) \sim  \sum_{k=1}^\infty \widetilde f_k(t),\text{ where   $\widetilde f_k\in \mathscr F_{m_*}(\widetilde n_k,-\alpha_k,\C^n)$ for  $k\in\N$.}
\eeq

Let $\mathcal S=\big\langle \{\alpha_k:k\in\N\}\big\rangle$ and $n_k$ be as in \eqref{nk}.
Then, similar to Scenario \ref{scen4},  we obtain the asymptotic expansion \eqref{fmu3} from \eqref{fal3}. In this case, it is clear that relation \eqref{mua1} is true.
\end{scenario}

\begin{theorem}\label{mainthm3}
Under Assumptions \ref{assumpf}, \ref{fLas2} and  \ref{assumpu},
 there exist  functions 
\beq\label{yFm} 
y_k\in \mathscr F_{m_*}(n_k,-\mu_k,\C^n)\text{ for $k\in\N$,}
\eeq  
 such that the solution $y(t)$ admits the asymptotic expansion
   \beq\label{yexplog}
  y(t)\sim \sum_{k=1}^\infty y_k(t) \text{ in the sense of Definition \ref{Lexpand}.} 
  \eeq
  
More specifically, suppose 
\beq \label{fkLm}
f_k(t)=p_k(\widehat \LL_{n_k}(t)) \text{ with $p_k\in \mathscr P_{m_*}(n_k,-\mu_k,\C^n)$  for all $k\in\N$.}
\eeq 
 Then the functions $y_k$'s in \eqref{yexplog} can be constructed recursively as follows. For each $k\in\N$,
$y_k(t)=q_k(\widehat \LL_{n_k}(t))$, where 
\beq \label{qk3}    
q_k  
=\mathcal Z_A\left( \sum_{m\ge 2}\ \sum_{\mu_{j_{1}}+\mu_{j_{2}}+\ldots \mu_{j_{m}}=\mu_k} \mathcal G_m( q_{j_{1}}, q_{j_{2}},\ldots ,  q_{j_{m}}) +p_k\right ).
\eeq  
\end{theorem}
\begin{proof}
We follow the proof of Theorem \ref{mainthm2}. Set $\psi(t)=\iln_{m_*}(t)$.
By the virtue of \eqref{dq1},  it holds, for $t\in (E_{n_k}(0),\infty)$ and $k\in\N$, that
\beq \label{ykeq3}    
y_k'(t)=\mathcal M_{-1} q_k(\widehat \LL_{n_k}(t))+\bigo(t^{-\gamma}) \quad\forall\gamma\in(0,1).
\eeq  

In \eqref{shorty}, by using \eqref{ykeq3} instead of \eqref{ykeq2}, we replace
$\sum_{k=1}^N \mathcal R q_k\circ \widehat\LL_{n_k}$ with $\bigo(t^{-\gamma})$.
Because a function of $\bigo(t^{-\gamma})$  is also of $\bigo(\psi(t)^{-\mu_{N+1}-\delta})$ for any $\delta>0$, then after neglecting \eqref{sumR}, all terms $\chi_k$'s for $1\le k\le N+1$ in calculations thereafter can be taken to be $0$.
The proof goes through. During the process,  formula  \eqref{ykdirect}  has the term $\chi_k$ dropped, and, hence, becomes \eqref{qk3}. 
\end{proof}

\section{Results for systems in the real linear spaces}\label{realsec}

In this section, we focus on the case when the system of ODEs is given in $\R^n$.
Naturally, it is expected that the (real-valued) decaying solutions can be approximated by real-valued functions. However, our constructions in the previous Sections \ref{eforce}--\ref{lforce} rely heavily on calculations with complex numbers. Hence, the approximating functions are not necessarily real-valued. Nonetheless, we will prove that it is, in fact, still true. 

We will use the idea of complexification, which we recall below in a brief and convenient form. For more details, see, e.g., \cite[section 77]{Halmos1974}.
 
Let $X$ be a linear space over $\R$. Its complexification is $X_\C=X+iX$ with the following natural addition and scalar multiplication. For any $z= x+iy$ and $z'=x'+iy'$ in $X_\C$ with $x,x',y,y'\in X$, and any $c=a+ib$ in $\C$ with $a,b\in\R$, define
\begin{align*}
z+z'&=(x+x')+i(y+y'),\\
cz&=(ax-by)+i(bx+ay).
\end{align*}

Then $X_\C$ is a linear space over $\C$ and, of course, $X\subset X_\C$. 

For $z=x+iy\in X_\C$, with $x,y\in X$,  its conjugate is defined by 
$\overline z=x-iy=x+i(-y).$
When more explicit notation is needed, we denote this $\overline z$ by ${\overline z}^{X_\C} $.
Obviously, 
$z+\overline z\in X$. Also, $z=\overline z$ if and only if $z\in X$. 
 One can verify that
\beq
 \overline{cz}=\bar c\, {\overline z} \text{ for all } c\in\C, z\in X_\C.
\eeq

Suppose $(X,\langle\cdot,\cdot\rangle_X)$ is an inner product  space over $\R$. Then $X_\C$ is an inner product space over $\C$ with the corresponding inner product $\langle\cdot,\cdot\rangle_{X_\C}$ defined by
\beqs
\langle x+iy,x'+iy'\rangle_{X_\C}=\langle x,x'\rangle_X+\langle y,y'\rangle_X+i(\langle y,x'\rangle_X-\langle x,y'\rangle_X) \text{ for }x,x',y,y'\in X.
\eeqs

Denote by $\|\cdot\|_X$ and $\|\cdot\|_{X_\C}$ the norms on $X$ and $X_\C$ induced  from their respective inner products. Then   
\beqs
\|x+iy\|_{X_\C}=(\|x\|_X^2+\|y\|_X^2)^{1/2}\text{ and } \|\overline z\|_{X_\C}=\| z\|_{X_\C} \text{  for all $x,y\in X$ and $z\in X_\C$.}
\eeqs

For any $k\in\N$,  the complexification of $X=\R^k$ is $\C^k$  as both a linear space and an inner product space. 
For $z=(z_1,\ldots,z_k)\in \C^k=X_\C$, $\overline z^{X_\C}$ is the standard conjugate vector $\bar z=(\bar z_1,\ldots,\bar z_k)$ and $\|z\|_{X_\C}$ is the standard Euclidean norm $|z|=(|z_1|^2+\ldots+|z_k|^2)^{1/2}$.


Let $S$ be a subset of $\C^k$ with $k\in\N$. We say $S$ preserves the conjugation if the conjugate $\bar z$ of any $z\in S$ also belongs to $S$. Define the conjugate set of $S$ to be $\overline S=\{\bar\lambda:\lambda\in S\}$.

Because we deal with real-valued solutions and forcing functions now, we need to restrict the classes of functions in  Definition \ref{newclass}.

\begin{definition}\label{realclass} 
Let $X$ be a linear space over $\R$, and  $X_\C$ be its complexification.
Define
\beqs 
\mathcal F_E(X_\C,X)=\{g\in \mathcal F_E(X_\C):g(t)\in X\text{ for all } t\in\R\}.
\eeqs 

For $\mu\in\R$, define
\beqs 
\mathcal F_E(\mu,X_\C,X)=\mathcal F_E(X_\C,X)\cap \mathcal F_E(\mu,X_\C).
\eeqs 
 \end{definition}

The next lemma provides the  characteristics for the representation \eqref{gEform} and the asymptotic expansions of the functions in $\mathcal F_E(X_\C,X)$. 

\begin{lemma}\label{invar1}
Let $X$ be a linear space over $\R$, and  $X_\C$ be its complexification.
The following statements hold true.
\begin{enumerate}[label=\tnum]
\item\label{iva}   Assume $g\in\mathcal F_E(X_\C)$. 
Then $g(t)\in X$ for all $t\in\R$
 if and only if 
 \beq\label{flam} g(t)=\sum_{\lambda\in S} g_\lambda(t)\text{ with }  g_\lambda(t)=p_\lambda(t) e^{\lambda t},
 \eeq 
  where $S$ is a finite subset of $\C$ that preserves the conjugation, $p_\lambda$'s are polynomials from $\R$ to $X_\C$, and 
  \beq \label{fbb}
  g_{\bar \lambda}=\overline{g_\lambda}\quad \text{ for all } \lambda\in S.
\eeq   

\item\label{ivb}  Assume, in addition, that $X$ is an inner product space over $\R$. Let $g:(T,\infty)\to X$ for some $T\in\R$. Suppose $g$, as a $X_\C$-valued function, has an asymptotic expansion, in the sense of Definition \ref{EEdef},
\beq\label{fe}
g(t)\sim \sum_{k=1}^\infty g_k(t),\quad g_k\in \mathcal F_E(-\mu_k,X_\C).
\eeq 
Then $g_k(t)\in X$ for all $k\in\N$ and $t\in\R$.
\end{enumerate}
\end{lemma}
\begin{proof}
We prove part \ref{iva} first. The sufficient condition is obvious thanks to the facts that $S$ preserves the conjugation and, by \eqref{fbb},
\beqs
g_\lambda(t)+g_{\bar\lambda}(t)=g_\lambda(t)+\overline{g_\lambda(t)} \in  X 
\text{ for all }\lambda\in S,\, t\in\R.
\eeqs 

We prove the necessary condition now. Consider a fucntion $g$ in $\mathcal F_E(X_\C)$  given by \eqref{flam} with a finite set $S$ first.
By replacing $S$ with the  union $S\cup \overline S$ as well as adding the zero functions if needed,
we can assume \eqref{flam} with $S$ preserving the conjugation.
Then 
$$\overline g=\sum_{\lambda\in S} \overline{g_\lambda}, \quad\text{with }  \overline{g_\lambda}(t)=\overline{ p_\lambda(t) }e^{\bar \lambda t}.$$ 

Because $g$ is $X$-valued, we have
\beq \label{flamb}
g=\overline g=\sum_{\lambda\in S}\overline{g_\lambda}.
\eeq 

By the uniqueness in Proposition \ref{unipres}, the two functions corresponding to $e^{\bar\lambda t}$ in \eqref{flam} and \eqref{flamb} must be the same, i.e., 
$g_{\bar\lambda}=\overline{g_\lambda}$ for all $\lambda\in S$.
Hence, we obtain \eqref{fbb}.

We prove part \ref{ivb} now.  
Clearly, the conjugate function $\overline g$ has the following asymptotic expansion
\beq\label{fbar}
\overline g(t)\sim \sum_{k=1}^\infty \overline{g_k}(t),
\eeq 
where $\overline{g_k}$ is the conjugate function of $g_k$, and, obviously, $\overline{g_k}\in \mathcal F_E(-\mu_k,X_\C)$ for all $k\in\N$.

Because $g$ is $X$-valued, we have $g=\overline g$, thus, it follows \eqref{fbar} that
\beq\label{feb}
g(t)\sim \sum_{k=1}^\infty \overline{g_k}(t),\quad \overline{g_k}\in \mathcal F_E(-\mu_k,X_\C).
\eeq 

By the uniqueness of the asymptotic expansion of $g$, see the second remark after Definition \ref{EEdef}, we deduce from \eqref{fe}   and \eqref{feb} that $g_k(t)=\overline{g_k}(t)$ for all $k\in\N$ and $t\in \R$. Hence, $g_k(t)\in X$  for all $k\in\N$ and $t\in \R$.
\end{proof}

Next, we have the complexification of multi-linear mappings.
For the sake of simplicity, we consider only the particular  space $\R^n$.

\begin{definition}
Let $\{e_j:1\le j\le n\}$ be the canonical basis for $\R^n$ and $\C^n$.
Let $M$ be an $m$-linear mapping from $(\R^n)^m$ to $\R^n$.
The complexification of $M$ is  $M_\C:(\C^n)^m\to \C^n$  defined by
\beq\label{MCdef}
\begin{aligned}
&M_\C\left(\sum_{j_1=1}^n z_{1,j_1} e_{j_1}, \sum_{j_2=1}^n z_{1,j_2} e_{j_2},\ldots, \sum_{j_m=1}^n z_{m,j_m} e_{j_m}\right )\\
& =\sum_{j_1,j_2,\ldots,j_m=1}^n  z_{1,j_1}z_{2,j_1} \ldots  z_{m,j_m}  M( e_{j_1},e_{j_2}, \ldots,  e_{j_m} )
\end{aligned}
\eeq
for all $z_{k,j_k}\in\C$, $k=1,2,\ldots,m$ and $j_k=1,2,\ldots,n$.
\end{definition}

Then $M_\C$ is the unique $m$-linear mapping (over $\C$) from $(\C^n)^m$ to $\C^n$ that extends  $M$.

Because $M( e_{j_1},e_{j_2}, \ldots,  e_{j_m} )$ in \eqref{MCdef} are real-valued, one has
\beq\label{MCbar}
M_\C(\bar z_1,\bar z_2,\ldots,\bar z_m)=\overline{M_\C(z_1,z_2,\ldots,z_m)} \quad\forall z_1,z_2,\ldots,z_m\in\C^n .
\eeq

We now impose the main assumptions for system \eqref{sys-eq} in $\R^n$.

\begin{assumption}\label{realassum}
The following are assumed to hold in this section.

\begin{enumerate}[label=\tnum]

\item The matrix $A$ is an $n\times n$  matrix of real numbers that satisfies Assumption \ref{assumpA}.

\item The function $G$ is from $\R^n$ to $\R^n$, satisfies Assumption \ref{assumpG} with 
 $G_m$'s being functions from $\R^n$ to $\R^n$.
 
 \item The multi-linear mappings $\mathcal G_m$'s in \eqref{GG} are from $(\R^n)^m$ to $\R^n$.
 
 \item The  forcing function $f(t)$ and solution $y(t)$ are $\R^n$-valued and satisfy Assumptions \ref{assumpf} and \ref{assumpu}.
 \end{enumerate}
\end{assumption}

In the following proofs, for the $m$-linear mapping $\mathcal G_m:(\R^n)^m\to \R^n$ given in Assumption \ref{realassum}, let  $\mathcal G_{m,\C}:(\C^n)^m\mapsto \C^n$ be its complexification.

\begin{theorem}\label{newthmE}
Theorem \ref{mainthm} holds true when we replace $ \mathcal F_E(-\mu_k,\C^n) $ in \eqref{fmu} and \eqref{yFE} with  $\mathcal F_E(-\mu_k,\C^n,\R^n)$.
\end{theorem}
\begin{proof}
We adjust the proof of Theorem \ref{mainthm}.
First, we replace \eqref{hypoeq} in the statement  ($\mathcal T_N$) by the following
\beq \label{ykC}    
y_k'+ Ay_k  
=\sum_{m\ge 2}\ \sum_{\mu_{j_{1}}+\mu_{j_{2}}+\ldots \mu_{j_{m}}=\mu_k} \mathcal G_{m,\C}(  y_{j_{1}}, y_{j_{2}},\ldots ,  y_{j_{m}}) +f_k,
\eeq  
 for $t\in\R$ and 
for $k=1,2,\ldots,N$.

In identity \eqref{GGm}, we can equate the last term with $\mathcal G_{m,\C}((u_N+v_N)^{(m)}).$
Then proceed the proof afterward with $\mathcal G_{m,\C}$ replacing $\mathcal G_{m}$.
At the end of that proof, we obtain the asymptotic expansion \eqref{yykex} for the $\R^n$-valued solution $y(t)$ with each $y_k\in \mathcal F_E(-\mu_k,\C^n)$ satisfying \eqref{ykC}.
By Lemma \ref{invar1}, each $y_k$ belongs to $\mathcal F_E(-\mu_k,\C^n,\R^n)$. 
Because $y_k(t)$ are real-valued now, the term  $\mathcal G_{m,\C}$ in \eqref{ykC} is equal to $\mathcal G_{m}$, and, hence, we obtain \eqref{ykeq}.
\end{proof}

It is worth pointing out that we do not complexify equation \eqref{sys-eq} in the above proof of Theorem \ref{newthmE}.
In fact, we do not have sufficient information about $G(y)$ to complexify it. 
We only approximate $G(y)$ by the finite sum $\sum_{m=2}^{M_{N+1}} \mathcal G_m(y^{(m)})$, then use the complexified mapping $\mathcal G_{m,\C}$ to replace $\mathcal G_m$ and continue the computations.

In dealing with power, logarithmic and iterated logarithmic functions valued in real linear spaces, we have the following counterpart of Definition \ref{Fclass}.

\begin{definition}\label{realF}
 $X$ be a linear space over $\R$, and $X_\C$ be its complexification.

\begin{itemize}
\item Define $\mathscr P(k,X_\C,X)$ to be set of functions of the form 
\beq\label{Rpzdef} 
p(z)=\sum_{\alpha\in S}  z^{\alpha}\xi_{\alpha}\text{ for }z\in (0,\infty)^{k+2},
\eeq 
where $S$ is a finite subset of $\C^{k+2}$ that preserves the conjugation,
and each $\xi_{\alpha}$ belongs to $X_\C$, with 
\beq\label{xiconj}
\xi_{\bar \alpha}=\overline{\xi_\alpha} \quad\forall \alpha\in S.
\eeq

\item Define $\mathscr P_{m}(k,\mu,X_\C,X)$ to be set of functions in $\mathscr P(k,X_\C,X)$ with the restriction that the set $S$ in \eqref{Rpzdef} is also a subset of $\mathcal E(m,k,\mu)$.

\item For $k\ge m\ge -1$ and $\mu\in\R$, define 
 \beqs 
 \mathscr F_{m}(k,\mu,X_\C,X)=\Big\{ p\circ \widehat{\LL}_{k}: p\in\mathscr P_{m}(k,\mu,X_\C,X)\Big\}.
 \eeqs
 \end{itemize}
 \end{definition}

Note, for any $t>0$ and $z\in \C$, that  $t^{\bar z}=\overline{t^z}$. 
Then, referring to \eqref{vecpow}, we have
\beq \label{zaconj}
 z^{\bar \alpha}=\overline{z^\alpha}  \text{ for all }z\in(0,\infty)^{k+2}, \ \alpha\in \C^{k+2}.
 \eeq 

Because of \eqref{zaconj}  and the conjugation condition \eqref{xiconj}, each function $p$ in the class $\mathscr P(k,X_\C,X)$ is, in fact, $X$-valued.

If we rewrite \eqref{Rpzdef} as
\beq
p(z)=\sum_{\alpha\in S} p_\alpha(z),\text{ where } p_\alpha(z)=z^\alpha \xi_\alpha,
\eeq
then condition \eqref{xiconj} is equivalent to 
\beq\label{newcri}
p_{\bar \alpha}(z)=\overline{p_\alpha(z)}  \text{ for all } z\in(0,\infty)^{k+2}, \ \alpha\in S.
\eeq

We remark that the classes $\mathscr P(k,X_\C,X)$  and  $\mathscr P_{m}(k,\mu,X_\C,X)$ are (additive) subgroups of $\mathscr P(k,X_\C)$, but not linear spaces over $\C$.

We examine the restrictions of the mappings $\mathcal M_j$'s, $\mathcal R$ and $\mathcal Z_A$ on the new classes in Definition \ref{realF}.

\begin{lemma}\label{invar2}
The following statements hold true.
\begin{enumerate}[label=\tnum]
\item\label{invi}  
	Each $\mathcal M_j$, for $-1\le j\le k$, maps $\mathscr P(k,\C^n,\R^n)$ into itself, 
	  $\mathcal R$ maps $\mathscr P(k,\C^n,\R^n)$, for $k\ge 0$, into itself, 
	and   $\mathcal Z_A$ maps $\mathscr P_{-1}(k,0,\C^n,\R^n)$, for $k\ge -1$,  into itself. 
\item\label{invii} 
	All $\mathcal M_j$'s, for $-1\le j\le k$,  and $\mathcal Z_A$ 
map  $\mathscr P_{m}(k,\mu,\C^n,\R^n)$ into itself for any integers $k\ge m\ge 0$ and real number $\mu$.
\item \label{inviii} 
	 $\mathcal R$ 
maps  $\mathscr P_{0}(k,\mu,\C^n,\R^n)$ into $\mathscr P_{0}(k,\mu-1,\C^n,\R^n)$ for any  $k\in\Z_+$ and $\mu\in\R$. 
\end{enumerate}
\end{lemma}
\begin{proof}
We prove part \ref{invi}, while  parts \ref{invii} and \ref{inviii} are consequences of \ref{invi} and properties \ref{R0}--\ref{R2} stated after Definition \ref{defMRZ}.

\medskip
We look into $\mathcal M_j$ first.
Let $p\in \mathscr P(k,\C^n,\R^n)$.
Assume $p(z)$ is given by \eqref{Rpzdef} with a finite set $S$  in $\C^{k+2}$ preserving the conjugation, and $\xi_\alpha\in \C^n$, $\xi_{\bar\alpha}=\overline{\xi_\alpha}$ for all $\alpha\in S$. 
Write
\beqs
\mathcal M_jp(z)=\sum_{\alpha\in S} f_\alpha(z),\text{ where } f_\alpha(z)=\alpha_j z^\alpha \xi_\alpha.
\eeqs 

Since
\beqs
f_{\bar \alpha}(z)=\bar \alpha_j z^{\bar \alpha} \xi_{\bar \alpha}
=\overline{ \alpha_j} \, \overline{z^{\alpha}}\, \overline{ \xi_\alpha}
=\overline{f_{\alpha}(z)},
\eeqs
we obtain, thanks to the equivalent criterion \eqref{newcri},  $\mathcal M_jp\in \mathscr P(k,\C^n,\R^n)$. 

\medskip
We now check with $\mathcal R$ defined by \eqref{chiz} for $k\ge 0$. We write
$\mathcal Rp=\sum_{j=0}^k \mathcal N_j p$, where $\mathcal N_j$, for each $0\le j\le k$, is the linear transformation  on $\mathscr P(k,\C^n)$ defined by
\beqs
(\mathcal N_j p)(z)=z_0^{-1}z_1^{-1}\ldots z_j^{-1}\mathcal M_jp(z) \text{ for } p\in \mathscr P(k,\C^n).
\eeqs

For  $p\in \mathscr P(k,\C^n,\R^n)$, we write $p$ as the finite sum $p=\sum_{\alpha} p_\alpha$, where
each $p_\alpha$ belongs to $\mathscr P(k,\C^n,\R^n)$ and is of the form
$p_\alpha(z)=z^\alpha\xi_\alpha+z^{\bar\alpha}\overline{ \xi_{\alpha}}$.
We have 
\beqs
\mathcal N_j p_\alpha(z)
=z_0^{-1}z_1^{-1}\ldots z_j^{-1}(\alpha_j z^\alpha \xi_{\alpha}
+  \bar\alpha_j z^{\bar \alpha} \overline{\xi_\alpha})
=g(z)+h(z),
\eeqs 
where 
\begin{align*}
g(z)&=\alpha_j z^\beta \xi_\alpha\quad\text{ with $\beta=\alpha-(e_0+e_1+\ldots+e_j)$, }\\
h(z)&=\bar\alpha_j z^\gamma \overline{ \xi_\alpha}\quad\text{ with $\gamma=\bar\alpha-(e_0+e_1+\ldots+e_j)$.}
\end{align*}
Above, $\{e_{-1},e_0,\ldots,e_k\}$ is the canonical basis of $\C^{k+2}$.

Note that $\gamma=\bar \beta$ and
$h(z)=\bar\alpha_j z^{\bar \beta} \overline{\xi_\alpha}=\overline{g(z)}$. Thus, by criterion \eqref{newcri}, $\mathcal N_jp_\alpha\in\mathscr P(k,\C^n,\R^n)$.

Then summing up over finitely many $\alpha$'s   yields $\mathcal N_jp\in\mathscr P(k,\C^n,\R^n)$.
Finally,   thanks to the group property  of $\mathscr P(k,\C^n,\R^n)$,  summing up in $j$ gives  $\mathcal Rp=\sum_{j=0}^k \mathcal N_j p\in\mathscr P(k,\C^n,\R^n)$.

\medskip
Consider formula \eqref{ZAp} of $\mathcal Z_Ap$ with $p\in \mathscr P_{-1}(k,0,\C^n,\R^n)$, for $k\ge -1$, given by \eqref{Rpzdef}--\eqref{xiconj} and $\alpha\in \mathcal E(-1,k,0)$ for any $\alpha\in S$.
We  write
\beq\label{ZpfR}
\mathcal Z_Ap(z)=\sum_{\alpha\in S} f_\alpha(z),\text{ where } f_\alpha(z)=z^\alpha(A+\alpha_{-1}I_n)^{-1}  \xi_\alpha.
\eeq 

Noticing that $A=\bar A$, one has
\begin{align*}
f_{\bar \alpha}(z)=z^{\bar \alpha}(A+\overline{ \alpha_{-1}}I_n)^{-1} \xi_{\bar \alpha}
&=\overline{z^{\alpha}}\, \overline{(A+\alpha_{-1}I_n)^{-1}}\, \overline{\xi_{\alpha}}
=\overline{f_\alpha(z)}.
\end{align*}

Together with formula \eqref{ZpfR} and criterion \eqref{newcri},  this implies $\mathcal Z_Ap\in\mathscr P_{-1}(k,0,\C^n,\R^n)$.
\end{proof}

\begin{theorem}\label{newthmPL}
We have the following counterparts of Theorems \ref{mainthm2}  and \ref{mainthm3}.

\begin{enumerate}[label=\tnum]
\item\label{newP} 
Theorem \ref{mainthm2} holds true when we replace $ \mathscr F_{0}(n_k,-\mu_k,\C^n)$ with $ \mathscr F_{0}(n_k,-\mu_k,\C^n,\R^n)$ in \eqref{fmu2} and \eqref{yF0}, and replace   $ \mathscr P_{0}(n_k,-\mu_k,\C^n)$ with $ \mathscr P_{0}(n_k,-\mu_k,\C^n,\R^n)$ in \eqref{fkLp}.

\item\label{newL} 
Theorem \ref{mainthm3}  holds true when we replace $ \mathscr F_{m_*}(n_k,-\mu_k,\C^n)$ with $ \mathscr F_{m_*}(n_k,-\mu_k,\C^n,\R^n)$ in \eqref{fmu3} and \eqref{yFm}, and replace   $ \mathscr P_{m_*}(n_k,-\mu_k,\C^n)$ with $ \mathscr P_{m_*}(n_k,-\mu_k,\C^n,\R^n)$ in \eqref{fkLm}.
\end{enumerate}
\end{theorem}
\begin{proof}
We prove part \ref{newP}. 
Assume \eqref{fkLp} with $p_k\in \mathscr P_{0}(n_k,-\mu_k,\C^n,\R^n)$ for all $k\in\N$.
By combining the proof of Theorem  \ref{mainthm2}  with the argument in the proof of Theorem \ref{newthmE}, we obtain asymptotic expansion \eqref{yykex2} with  \eqref{ykLq}, where 
\beq \label{newqk}
q_k  
=\mathcal Z_A\left( \sum_{m\ge 2}\ \sum_{\mu_{j_{1}}+\mu_{j_{2}}+\ldots \mu_{j_{m}}=\mu_k} \mathcal G_{m,\C}( q_{j_{1}}, q_{j_{2}},\ldots ,  q_{j_{m}}) +p_k-\chi_k\right ),
\eeq  
for $\chi_k$ being given by \eqref{chikdef}. That is, $\mathcal G_{m,\C}$ in \eqref{newqk} replaces $\mathcal G_m$ in \eqref{ykdirect}.

We examine the construction of $q_k$'s in \eqref{newqk}. 
By property \ref{8a} after Theorem \ref{mainthm2}, $q_1=\mathcal Z_Ap_1$.
Because $p_1\in  \mathscr P_{0}(n_1,-\mu_1,\C^n,\R^n)$, then by applying  Lemma \ref{invar2}, we have $q_1\in  \mathscr P_{0}(n_1,-\mu_1,\C^n,\R^n)$.

Let $k\ge 2$. Suppose $q_j\in\mathscr P_{0}(n_j,-\mu_j,\C^n,\R^n)$ for $1\le j\le k-1$.

In \eqref{newqk}, we already know 
$p_k\in\mathscr P_{0}(n_k,-\mu_k,\C^n,\R^n)$.

By Lemma \ref{invar2}, $\mathcal R q_\lambda$ in \eqref{chikdef} belongs to $\mathscr P_{0}(n_k,-\mu_k,\C^n,\R^n)$. Therefore,   $\chi_k$  in \eqref{newqk}  belongs to $\mathscr P_{0}(n_k,-\mu_k,\C^n,\R^n)$.

Consider  each $\mathcal G_{m,\C}( q_{j_{1}}, q_{j_{2}},\ldots ,  q_{j_{m}}) $ in \eqref{newqk}.
By the distributive property of $\mathcal G_{m,\C}$ with respect to the addition and the form \eqref{Rpzdef} of each $q_{j_{\ell}}$, it suffices to examine the following pairs
$$g(z)=\mathcal G_{m,\C}( z^{\alpha_{(1)}} \xi_{\alpha_{(1)}}, z^{\alpha_{(2)}} \xi_{\alpha_{(2)}}, \ldots,
z^{\alpha_{(m)}} \xi_{\alpha_{(m)}}) $$
and 
$$h(z)=\mathcal G_{m,\C}( z^{\overline{\alpha_{(1)}}} \xi_{\overline{\alpha_{(1)}}}, z^{\overline{\alpha_{(2)}}} \xi_{\overline{\alpha_{(2)}}}, \ldots,
z^{\overline{\alpha_{(m)}} }\xi_{\overline{\alpha_{(m)}}}) .$$

We calculate
$$g(z)=z^\beta \mathcal G_{m,\C}( \xi_{\alpha_{(1)}}, \xi_{\alpha_{(2)}}, \ldots,\xi_{\alpha_{(m)}}) , $$
where
$\beta=\alpha_{(1)} +\alpha_{(2)}+\ldots+\alpha_{(m)}$,
and 
$$h(z)=z^\gamma \mathcal G_{m,\C}( \xi_{\overline{\alpha_{(1)}}},  \xi_{\overline{\alpha_{(2)}}}, \ldots,\xi_{\overline{\alpha_{(m)}}}),$$
where
$\gamma={\overline{\alpha_{(1)}}}+{\overline{\alpha_{(2)}}}+\ldots+{\overline{\alpha_{(m)}}}$.

Clearly, $\gamma=\bar\beta$, and by \eqref{MCbar},
\beqs
h(z)=z^{\bar\beta}\mathcal G_{m,\C}(\overline{ \xi_{\alpha_{(1)}}}, \overline{ \xi_{\alpha_{(2)}}}, \ldots,\overline{\xi_{\alpha_{(m)}}})   
=\overline{ z^\beta}\,\, \overline{\mathcal G_{m,\C}( \xi_{\alpha_{(1)}}, \xi_{\alpha_{(2)}}, \ldots,\xi_{\alpha_{(m)}})}
=\overline{g(z)}.
\eeqs

Thus, $\mathcal G_{m,\C}( q_{j_{1}}, q_{j_{2}},\ldots ,  q_{j_{m}}) $ belongs to $ \mathscr P_{0}(n_k,-\mu_k,\C^n,\R^n)$. 

Summing up, we obtain
\beqs
\sum_{m\ge 2}\ \sum_{\mu_{j_{1}}+\mu_{j_{2}}+\ldots \mu_{j_{m}}=\mu_k} \mathcal G_{m,\C}( q_{j_{1}}, q_{j_{2}},\ldots ,  q_{j_{m}}) +p_k-\chi_k \in  \mathscr P_{0}(n_k,-\mu_k,\C^n,\R^n). 
\eeqs

Applying $\mathcal Z_A$ to this element and using Lemma \ref{invar2}, we have $q_k\in \mathscr P_{0}(n_k,-\mu_k,\C^n,\R^n)$. 

By the Induction Principle, $q_k\in \mathscr P_{0}(n_k,-\mu_k,\C^n,\R^n)$ for all $k\in\N$. 

Now that all $q_k$'s are real-valued, we can replace $\mathcal G_{m,\C}$ in \eqref{newqk} with $\mathcal G_m$ and obtain \eqref{ykdirect}.
This completes the proof of part \ref{newP}.

The proof of  part \ref{newL} is similar by neglecting the terms $\chi_k$'s.
\end{proof}

In the above proof of Theorem \ref{newthmPL},  formulas \eqref{ykdirect} and \eqref{newqk} are the same. However, fornula \eqref{newqk} is preferred in manipulating the complex powers in order to perform the operator $\mathcal Z_A$.

In Theorems \ref{newthmE} and \ref{newthmPL}, the functions in the asymptotic expansions are still expressed by the use of complex numbers. Below, we will remove such expressions and write the results in terms of real-valued functions only.

\begin{definition}\label{polyS}
Let $X$ be a linear space over the field $\K=\R$ or $\K=\C$.
\begin{enumerate}[label=\tnum]
 \item A function $g:\R\to X$ is an $X$-valued S-polynomial if it is a finite sum of the functions in the set
\beqs
\Big \{ t^m \cos(\omega t)Z,\ t^m \sin(\omega t) Z: m\in\Z_+,\ \omega\in\R, \ Z\in X\Big\}.
\eeqs

\item Denote by $\mathcal F_0(X)$, respectively, $\mathcal F_1(X)$ the set of all $X$-valued polynomials, respectively, S-polynomials.
\end{enumerate}
\end{definition}

By the virtue of Lemma \ref{invar1}\ref{iva} and Euler's formula, one has
\beq\label{FF}
\mathcal F_E(0,\C^n,\R^n)=\mathcal F_1(\R^n).
\eeq

Let $(X,\|\cdot\|_X)$ be a normed space,
and  $(\gamma_k)_{k=1}^\infty$ be the same as in Definition \ref{EEdef}\ref{EE1}. we define the asymptotic expansion
\beq\label{realEx}
g(t)\sim \sum_{k=1}^\infty \widehat g_k(t)e^{-\gamma_k t}, \text{ where } \widehat g_k\in \mathcal F_1(X) \text{ for }k\in\N,
\eeq
in the same way as  \eqref{expan2}--\eqref{gdrem2}.

Therefore, one can use relation \eqref{FF} and asymptotic expansion \eqref{realEx} to restate Theorem \ref{newthmE} as the following.

\begin{theorem}\label{thmE3}
Let $(\mu_k)_{k=1}^\infty$ be the same as in Assumption \ref{fEas}.
If $f(t)$ has the following asymptotic expansion
\beq\label{freal}
f(t)\sim \sum_{k=1}^\infty \widehat f_k(t)e^{-\mu_k t}, \text{ where } \widehat f_k\in \mathcal F_1(\R^n) \text{ for }k\in\N,
\eeq
then the solution $y(t)$ admits the asymptotic expansion
\beq\label{yreal}
y(t)\sim \sum_{k=1}^\infty \widehat q_k(t)e^{-\mu_k t},  \text{ with }\widehat q_k\in \mathcal F_1(\R^n) \text{ for }k\in\N.
\eeq
\end{theorem}

Note that even if $\widehat f_k\in \mathcal F_0(\R^n)$ for all $k\in\N$ in \eqref{freal}, we can only conclude $\widehat q_k\in \mathcal F_1(\R^n)$ in \eqref{yreal}.
 
We turn to Theorem \ref{newthmPL} now. 
We will characterize the classes $ \mathscr F_{m}(k,0,\C^n,\R^n)$  more explicitly using real-valued  functions only.

\begin{definition}\label{realPL}
Given integers $k\ge m\ge 0$.
Define the class $\mathcal P_m^1(k,\R^n)$ to be the collection of functions which are the finite sums of the following functions
\beq\label{realpz}
z=(z_{-1},z_0,\ldots,z_k)\in (0,\infty)^{k+2}\mapsto z^\alpha \prod_{j=0}^k \sigma_j(\omega_j z_j)\xi,
\eeq
where $\xi\in\R^n$, 
 $\alpha\in \mathcal E(m,k,0)\cap\R^{k+2}$, 
$\omega_j$'s are real numbers,  
and, for each $j$,  either $\sigma_j=\cos$ or $\sigma_j=\sin$.

Define the class $\mathcal P_m^0(k,\R^n)$ to be the subset of $\mathcal P_m^1(k,\R^n)$ when all $\omega_j$'s in \eqref{realpz} are zero. 
\end{definition}

The results with the classes $\mathcal P_m^0(k,\R^n)$, for a real diagonalizable matrix $A$, were established in \cite{CaH3}. We focus on the classes $\mathcal P_m^1(k,\R^n)$ here.

Note in \eqref{realpz} that $\Re(\alpha_{-1})=\Re(\alpha_0)=\ldots=\Re(\alpha_m)=0$. Hence, the class $\mathcal P_m^1(k,\R^n)$  is related to $\mathscr P_{m}(k,0,\C^n,\R^n) $.

Let $m\in \Z_+$, $k\ge m$, $-1\le j\le k$ and  $\omega\in\R$. For $\xi=x+iy\in\C^n$ with $x,y\in \R^n$, one has
\beqs
\iln_j(t)^{i\omega}\xi+\iln_j(t)^{-i\omega}\bar \xi
=2(\cos(\omega \iln_{j+1}(t))x-\sin(\omega \iln_{j+1}(t))y).
\eeqs

Consequently, one can prove by induction that if $p\in \mathscr P_{m}(k,0,\C^n,\R^n) $ then
\beq\label{PPR2}
p\circ  \widehat \LL_k=q\circ  \widehat \LL_{k+1}\text{ for some }q\in \mathcal P_{m}^1(k+1,\R^n).
\eeq

In particular, 
\beq\label{spec4}
q\in \mathcal P_{m}^1(k,\R^n) \text{ provided } 
p(z)=\sum z^\alpha \xi_\alpha, \text{ where $\alpha=(\alpha_{-1},\alpha_0,\ldots,\alpha_{k})$ with $\Im(\alpha_{k})=0$.} 
\eeq
(For, there is no term $\iln_{k+1}(t)^\alpha$ in $p\circ  \widehat \LL_k(t)$, and there is no term $\iln_k(t)^{i\omega}$ in $p\circ  \widehat \LL_k(t)$ to contribute to $\cos(\omega \iln_{j+1}(t))$ and $\sin(\omega \iln_{j+1}(t))$ in $q\circ  \widehat \LL_k(t)$.)

We now observe that
\beqs
 \cos(\omega\iln_j(t))=\frac12\left(e^{i\omega\iln_j(t)}+e^{-i\omega\iln_j(t)}\right)
 =\frac12 \left(\iln_{j-1}(t)^{i\omega}+\iln_{j-1}(t)^{-i\omega }\right),
 \eeqs
and, similarly,
 \beqs
 \sin(\omega\iln_j(t))=\frac1{2i}\left( \iln_{j-1}(t)^{i\omega}- \iln_{j-1}(t)^{-i\omega }\right).
 \eeqs
 
 Therefore, 
 \beq\label{sincos1}
  \cos(\omega\iln_j(t))=g(\widehat \LL_k(t))\text{ and }\sin(\omega\iln_j(t))=h(\widehat \LL_k(t))
   \eeq
 where
 \beq\label{sincos2} 
 g(z)=\frac12(z_{j-1}^{i\omega}+z_{j-1}^{-i\omega})\text{ and }
h(z)=\frac1{2i}(z_{j-1}^{i\omega}-z_{j-1}^{-i\omega}).
\eeq

Obviously,
\beq\label{sincos3} 
g,h\in \mathscr P_{m}(k,0,\C,\R).
\eeq

Using properties \eqref{sincos1}, \eqref{sincos2}, \eqref{sincos3} and the same arguments in Theorem \ref{newthmPL} to  prove $\mathcal G_{m,\C}( q_{j_{1}}, q_{j_{2}},\ldots ,  q_{j_{m}}) $ belongs to $ \mathscr P_{0}(n_k,-\mu_k,\C^n,\R^n)$, one can verify that
if $p\in \mathcal P_{m}^1(k,\R^n) $ then
\beq\label{PPR1}
p\circ  \widehat \LL_k=q\circ  \widehat \LL_k\text{ for some }q\in \mathscr P_{m}(k,0,\C^n,\R^n).
\eeq

More specifically, 
\beq\label{spec3}
q(z)=\sum z^\alpha \xi_\alpha, \text{ where $\alpha=(\alpha_{-1},\alpha_0,\ldots,\alpha_{k})$ with $\Im(\alpha_{k})=0$.} 
\eeq

The last condition is due to the fact that  the functions  $\cos(\omega\iln_k(t))$ and $\sin(\omega\iln_k(t))$ can be converted via  \eqref{sincos1} and \eqref{sincos2}, when $j=k$,  using the functions of   the variable $z_{k-1}$.

Let $m_*$, $(\gamma_k)_{k=1}^\infty$ and $(n_k)_{k=1}^\infty$ be the same as in Definition \ref{Lexpand}\ref{LE1}.
We say a function $g:(T,\infty)\to\R^n$, for some $T\in\R$, has an asymptotic expansion
\beq\label{realLE}
g(t)\sim \sum_{k=1}^\infty \widehat g_k(\widehat\LL_{n_k}(t))\iln_{m_*}(t)^{-\gamma_k}, \text{ where $\widehat g_k\in \mathcal P_{m_*}^1(n_k,\R^n) $ for $k\in\N$, }
\eeq
if, for each $N\in\N$, there is some $\mu>\gamma_N$ such that
\beqs
\left |g(t) - \sum_{k=1}^N \widehat g_k(\widehat\LL_{n_k}(t))\iln_{m_*}(t)^{-\gamma_k}\right |=\bigo(\iln_{m_*}(t)^{-\mu}).
\eeqs

We can now  restate Theorem \ref{newthmPL} using the class $\mathcal P_m^1(k,\R^n)$ and the expansion form \eqref{realLE} as the following.

\begin{theorem}\label{thmPL3}
Given $m_*\in\Z_+$. Let $(\mu_k)_{k=1}^\infty$ and $(n_k)_{k=1}^\infty$ be the same as in Assumption \ref{fLas} if $m_*=0$, and  be the same as in Assumption \ref{fLas2} if $m_*\ge 1$.
If $f(t)$ has the asymptotic expansion
\beq\label{freal2}
f(t)\sim \sum_{k=1}^\infty \widehat p_k(\widehat \LL_{n_k}(t))\iln_{m_*}(t)^{-\mu_k}, 
\text{ where } \widehat p_k\in \mathcal P_{m_*}^1(n_k,\R^n) \text{ for }k\in\N,
\eeq
then the solution $y(t)$ admits the asymptotic expansion
\beq\label{yreal2}
y(t)\sim \sum_{k=1}^\infty \widehat q_k(\widehat \LL_{n_k}(t))\iln_{m_*}(t)^{-\mu_k}, 
 \text{ with }\widehat q_k\in \mathcal P_{m_*}^1(n_k,\R^n)  \text{ for }k\in\N.
\eeq
\end{theorem}
\begin{proof}
For each $k\in\N$, thanks to \eqref{PPR1} we have $\widehat p_k(\widehat \LL_{n_k}(t))=\widetilde p_k(\widehat \LL_{n_k}(t))$ for some 
$\widetilde p_k\in \mathscr P_{m_*}(n_k,0,\C^n,\R^n) $.
Applying Theorem \ref{newthmPL}, we obtain the asymptotic expansion
\beqs
y(t)\sim \sum_{k=1}^\infty \widetilde q_k(\widehat \LL_{n_k}(t))\iln_{m_*}(t)^{-\mu_k}, 
 \text{ where }\widetilde q_k\in  \mathscr P_{m_*}(n_k,0,\C^n,\R^n)   \text{ for }k\in\N.
\eeqs

Thanks to property \eqref{PPR2}, we have 
$$ \widetilde q_k(\widehat \LL_{n_k}(t))=\widehat q_k(\widehat \LL_{n_k+1}(t)) \text{ for some }\widehat q_k\in \mathcal P_{m_*}^1(n_k+1,\R^n).$$

We examine $\widehat q_k$ more closely. In fact, $\widetilde p_k$ has the  representation as in \eqref{spec3} with
$\alpha=(\alpha_{-1},\ldots,\alpha_{n_k})$  satisfying $\Im(\alpha_{n_k})=0$. By the recursive formula \eqref{ykdirect} for $m_*=0$ or \eqref{qk3} for $m_*\ge 1$, each $\widetilde q_k$ has the same property. By the virtue of \eqref{spec4}, we have $\widehat q_k\in \mathcal P_{m_*}^1(n_k,\R^n)$,  and hence, obtain  \eqref{yreal2}.
\end{proof}

Similar to the remark after Theorem \ref{thmE3}, even if $\widehat p_k\in \mathcal P_{m_*}^0(n_k,\R^n) $ in \eqref{freal2}, we, in general, can only have
$\widehat q_k\in \mathcal P_{m_*}^1(n_k,\R^n)$ in \eqref{yreal2}.

\medskip
Our general results -- Theorems \ref{mainthm}, \ref{mainthm2}, \ref{mainthm3}, \ref{newthmE}, \ref{newthmPL}, \ref{thmE3} and \ref{thmPL3} -- can have more specific forms in many different situations, see, for instance, \cite[Example 5.9--Example 5.12]{CaH3}.
Below, we demonstrate the last theorem with some examples (for the systems in $\R^n$.)

\begin{example}
If, for large $t>0$,
   \beqs
 f(t)=\frac{\cos(\alpha t)(\ln t) (\ln\ln t)^{-1/3}}{t^m}\xi \text{ for some $m\in\N$ and $\xi\in \R^n$,}
 \eeqs
 then the solution $y(t)$ admits the asymptotic expansion
 \beqs
 y(t)\sim \sum_{k=0}^\infty \frac{q_k(t)}{t^{m+k}},
 \eeqs
 where $q_k(t)=\widehat q_k(\widehat \LL_2(t))$ with $\widehat q_k\in \mathcal P_0^1(2,\R^n)$.
 Roughly speaking, the functions $q_k(t)$'s are composed by 
 \beq \label{basicf}
 \cos(\omega \iln_j(t)),\ 
 \sin(\omega \iln_j(t)),\  
 \iln_\ell(t)^\alpha,
 \eeq 
 for $j=0,1,2$ and $\ell=1,2$,  with some real numbers $\omega$'s and $\alpha$'s.
 \end{example}

\begin{example}
If, for large $t>0$, \beqs
 f(t)=\frac{ \cos(2 t) \sin(3 \ln\ln t) (\ln\ln\ln t)^2 \sin(5 \ln\ln\ln t)} {(\ln t)^{1/2}} \xi \text{ for some  $\xi\in \R^n$,}
 \eeqs
  then the solution $y(t)$ admits the asymptotic expansion
 \beqs
 y(t)\sim \sum_{k=1}^\infty \frac{q_k(t)}{(\ln t)^{k/2}},
 \eeqs
 where, roughly speaking,  $q_k(t)$'s are functions composed by the functions in \eqref{basicf} for $j=0,1,2,3$ and $\ell=2,3$.
\end{example}

\appendix
\section{}\label{append}

\begin{proof}[Proof of Lemma \ref{uniE}]
This proof follows \cite[Lemmas 2.3 and A.1]{HTi1}.

Suppose the conclusion is not true. Denote $s_\lambda=p_\lambda-q_\lambda$. Then $s_\lambda\ne 0$ for some $\lambda\in S$.

We write each $\lambda\in S$ as $\lambda=\mu+i\omega_\lambda$ with $\omega_\lambda\in\R$. We have from \eqref{gminush} that
\beq\label{limsl}
\lim_{t\to\infty} \sum_{\lambda\in S} s_\lambda(t) e^{i\omega_\lambda t}=0.
\eeq

Take $d_*$ to be the maximum of the degrees of the polynomials $s_\lambda$'s, for $\lambda\in S$. For $\lambda\in S$, write 
$$s_\lambda(t)=t^{d_*} \xi_\lambda + \text{ lower powers of $t$, with $\xi_\lambda\in X$.}$$

 Note that 
 \beq\label{xis} \exists \lambda\in S: \xi_\lambda\ne 0.
 \eeq 
 
  Dividing \eqref{limsl} by $t^{d_*}$ yields 
\beq\label{lims2}
\lim_{t\to\infty} \sum_{\lambda\in S} \xi_\lambda e^{i\omega_\lambda t}=0.
\eeq

Let $\{Y_j: 1\le j\le N\}$, for some $N\ge 1$, be a basis of the linear span of $\{\xi_\lambda:\lambda\in S\}$.

Write $\xi_\lambda=\sum_{j=1}^N c_{\lambda,j}Y_j$ where $c_{\lambda,j}$ are complex numbers.
For each $j=1,2,\ldots,N$, we have from \eqref{lims2} that
\beq\label{lims3}
\lim_{t\to\infty} \sum_{\lambda\in S} c_{\lambda,j} e^{i\omega_\lambda t}=0.
\eeq

\noindent\textbf{Claim.} $c_{\lambda,j}=0$ for all $\lambda\in S$.

Accepting this Claim momentarily, we then have, for each $\lambda\in S$, that  $c_{\lambda,j}=0$ for all $j$. Hence, $\xi_\lambda=0$ for all $\lambda\in S$, which contradicts \eqref{xis}. Therefore, the conclusion of Lemma \ref{uniE} must be true.

\medskip
Fix an integer $j\in[1,N]$, we prove the Claim from \eqref{lims3}.

Let $\omega_*=1+\max\{|\omega_\lambda|:\lambda\in S\}$. Multiplying \eqref{lims3} by $e^{i\omega_* t}$, we
rephrase the problem as the following. There are complex numbers $a_k$'s and strictly increasing positive numbers  $r_k$'s, for $1\le k\le m$, 
 such that
\beq\label{aklim}
\lim_{t\to\infty}\sum_{k=1}^m a_{k}e^{i r_k t}= 0,
\eeq
\beq\label{caset} \{c_{\lambda,j}:\lambda\in S\}=\{a_k:1\le k\le m\},\eeq 
and
 \beqs \{\omega_\lambda+\omega_*:\lambda\in S\}=\{r_k:1\le k\le m\}.\eeqs 

We will establish that 
\beq \label{akzero}
a_k=0\text{  for all }k=1,2,\ldots,m.
\eeq 

In fact, we consider a more general statement, namely,

($H_m$) ``\textit{If \eqref{aklim} holds for  complex numbers $a_k$'s and strictly increasing positive numbers  $r_k$'s, for $1\le k\le m$,  then \eqref{akzero} is true.}"

We prove, by induction, that ($H_m$) is true for all $m\in\N$.

Suppose $m=1$ and \eqref{aklim} holds. Then it is clear that $a_1=0$. Therefore, ($H_1$) is correct.

Let $m\ge 1$. Suppose ($H_m$) holds true.
Now, assume 
\beq\label{akmplus}
\lim_{t\to\infty}\sum_{k=1}^{m+1} a_{k}e^{i r_k t}= 0,
\eeq
for some complex numbers $a_k$'s and strictly increasing positive numbers  $r_k$'s, for $1\le k\le m+1$.

Integrating the sum in \eqref{akmplus} from $t$ to $t+2\pi/r_{m+1}$ and taking $t\to\infty$ give
\beqs
\lim_{t\to\infty} \left(\sum_{j=1}^{m}\frac{a_{k}}{ir_k}e^{i r_k t}(e^{2\pi i r_k/r_{m+1}}-1)  +0\right)=0.
\eeqs

By the Induction Hypothesis ($H_m$), we have 
\beq\label{newak}
\frac{a_{k}}{i r_k}(e^{2\pi i r_k/r_{m+1}}-1)=0 \text{ for } 1\le k\le m.
\eeq

Note in \eqref{newak} that $e^{2\pi i r_k/r_{m+1}}\ne 1$. Hence, $a_k=0$  for $1\le k\le m$.
Returning to \eqref{akmplus}, we then have 
\beqs
\lim_{t\to\infty}(a_{m+1}e^{i r_{m+1} t})= 0,
\eeqs
which yields $a_{m+1}=0$. Therefore, we obtain ($H_{m+1}$).

By the Induction Principle, the statement ($H_m$)  is true for all $m\in\N$.

Consequently, we have \eqref{akzero}, which,
together with  \eqref{caset}, implies that the Claim is true.

The proof of Lemma \ref{uniE} is complete.
\end{proof}

\begin{proof}[Proof of Proposition \ref{unipres}]
Because $X$ is not assumed to be a normed space, we cannot apply Lemma \ref{uniE} straightforwardly.

\textit{Part 1.} Suppose  $g$ has two representations
 \beq\label{ggsum}
g(t)=\sum_{\lambda\in S} p_\lambda(t)e^{\lambda t}\text{ and }
g(t)=\sum_{\lambda\in S} q_\lambda(t)e^{\lambda t},   \text{ for }t\in\R.
\eeq

By subtracting the above two formulas of $g(t)$, it suffices to assume now $g=0$ has the form \eqref{gEform}, and then prove that $p_\lambda=0$ for all $\lambda\in S$.

Suppose $p_\lambda(t)=\sum_{j=0}^{d_\lambda}t^j\xi_{\lambda,j}$ for $\lambda\in S$, and some vectors $\xi_{\lambda,j}$'s in $X$.
Let $E$ be the finite dimensional linear subspace of $X$ spanned by the vectors $\xi_{\lambda,j}$'s for $\lambda\in S$ and $0\le j\le d_\lambda$.
Let $\{Y_k:1\le k\le N\}$ be a basis of $E$. For any $\lambda\in S$ and  integer $j\in[1, d_\lambda]$, one has
$\xi_{\lambda,j}=\sum_{k=1}^N c_{\lambda,j,k}Y_k$ for some complex numbers $c_{\lambda,j,k}$.
We then have, for all $t\in \R$,
\beq
0=\sum_{\lambda\in S}\sum_{j=1}^{d_\lambda}e^{\lambda t} t^j\xi_{\lambda,j}
=\sum_{\lambda\in S}\sum_{j=1}^{d_\lambda}\sum_{k=1}^N e^{\lambda t} t^j  c_{\lambda,j,k}Y_k
=\sum_{k=1}^N \Big(\sum_{\lambda\in S}Q_{\lambda,k}(t)e^{\lambda t} \Big)Y_k,
\eeq
where $Q_{\lambda,k}(t)=\sum_{j=1}^{d_\lambda} t^j  c_{\lambda,j,k}$.
This implies, for each $k\in[1,N]$ and all $t\in \R$, that
\beq\label{Qsum}
\sum_{\lambda\in S}Q_{\lambda,k}(t)e^{\lambda t}=0.
\eeq

Note that each $Q_{\lambda,k}$ is a polynomial from $\R$ to $\C$.
Suppose $\Re S=\{\mu_1>\mu_2>\ldots>\mu_s\}$. 
We decompose 
$$S=\bigcup_{m=1}^s S_m\text{ with }S_m=\{\lambda\in S:\Re\lambda=\mu_m\}.$$

We then write the sum in \eqref{Qsum} as $\sum_{\lambda\in S}=\sum_{\lambda\in S_1}+\sum_{\lambda\in S_2}+\ldots+\sum_{\lambda\in S_m}$.
Dividing \eqref{Qsum} by $e^{\mu_1 t}$ and taking $t\to\infty$ give
\beq
\lim_{t\to\infty} e^{-\mu_1 t}\sum_{\lambda\in S_1}Q_{\lambda,k}(t)e^{\lambda  t}=0.
\eeq

Applying Lemma \eqref{uniE} to $X=\C$, $S=S_1$ and $p_\lambda=Q_{\lambda,k}$, $\mu=\mu_1$, $q_\lambda=0$, we have the polynomials $Q_{\lambda,k}=0$ for all $\lambda\in S_1$.
With this fact, the sum in \eqref{Qsum} is reduced to 
$\sum_{\lambda\in S_2}+\ldots+\sum_{\lambda\in S_m}$.
Repeating the above arguments with $S_2$ replacing $S_1$, we obtain $Q_{\lambda,k}=0$ for all $\lambda\in S_2$. By this recursive reasoning, we obtain $Q_{\lambda,k}=0$ for all $\lambda\in S_m$ for $1\le m\le s$, that is, 
$Q_{\lambda,k}=0$ for all $\lambda\in S$.
Because each $Q_{\lambda,k}(t)$ is a polynomial, this yields that its coefficients $c_{\lambda,j,k}$'s, for $0\le j\le d_\lambda$, are zeros.
Now that $c_{\lambda,j,k}=0$ for all $\lambda,j,k$, we infer $\xi_{\lambda,j}=0$ for all $\lambda,j$, and hence $p_\lambda=0$ for all $\lambda$.

\textit{Part 2.}  Now, assume $g$ has two representations
 \beqs
g(t)=\sum_{\lambda\in S'} p_\lambda(t)e^{\lambda t}\text{ and }
g(t)=\sum_{\lambda\in S''} q_\lambda(t)e^{\lambda t},   \text{ for }t\in\R,
\eeqs
with $p_\lambda\ne 0$ for all $\lambda\in S'$ and $q_\lambda\ne 0$ for all $\lambda\in S''$. 

Set $S=S'\cup S''$, define $p_\lambda=0$ for $\lambda\in S\setminus S'$, and $q_\lambda=0$ for $\lambda\in S\setminus S''$. Then we have \eqref{ggsum}. Thanks to  Part 1, $p_\lambda=q_\lambda$ for all $\lambda\in S$. Thus, $S'=S''=S$. 
\end{proof}

\bibliographystyle{abbrv}
\def\cprime{$'$}

 \end{document}